\documentclass[preprint,12pt]{elsarticle}

\usepackage{amsmath,amssymb,amsthm,mathtools,mathrsfs}
\usepackage{bm}
\usepackage{xcolor}
\usepackage{tikz}
\usetikzlibrary{positioning,arrows.meta,calc}
\usepackage{float}
\usepackage{graphicx}
\usepackage{enumitem}
\usepackage{microtype}
\usepackage{aliascnt}
\usepackage[hidelinks]{hyperref}
\usepackage[nameinlink,capitalize,noabbrev]{cleveref}
\journal{Analysis \& PDE}

\newtheorem{theorem}{Theorem}[section]

\newaliascnt{proposition}{theorem}
\newtheorem{proposition}[proposition]{Proposition}
\aliascntresetthe{proposition}

\newaliascnt{lemma}{theorem}
\newtheorem{lemma}[lemma]{Lemma}
\aliascntresetthe{lemma}

\newaliascnt{corollary}{theorem}
\newtheorem{corollary}[corollary]{Corollary}
\aliascntresetthe{corollary}

\newaliascnt{remark}{theorem}
\newtheorem{remark}[remark]{Remark}
\aliascntresetthe{remark}

\theoremstyle{definition}
\newaliascnt{definition}{theorem}
\newtheorem{definition}[definition]{Definition}
\aliascntresetthe{definition}

\newcommand{\Rh}{R}
\newcommand{\T}{T}
\newcommand{\M}{M}

\newcommand{\norm}[1]{\left\lVert #1\right\rVert}

\newcommand{\Op}{\operatorname{Op}}

\begin{document}

\begin{frontmatter}

\title{Uniform Eigenfunction Observability under Mixed Reflection Monodromy}

\author[aff1]{Binh T. Nguyen\corref{cor1}}
\ead{ngtbinh@hcmus.edu.vn}
\cortext[cor1]{Corresponding author}
\address[aff1]{Faculty of Mathematics and Computer Science, University of Science, Vietnam National University Ho Chi Minh City, Ho Chi Minh City, Vietnam}

\begin{abstract}
We study uniform observability for Laplace eigenfunctions with mixed boundary
conditions whose reflection signs do not define a scalar character. For the
DDN/NND sectors of the equilateral rhombus, the resulting \(\mathbb Z_2\)
monodromy is resolved by a degree-two branched arithmetic translation surface
of genus two. On an explicit class of admissible open sets, every exact
mixed-sector eigenfunction satisfies a local \(L^2\) lower bound uniform in the
eigenvalue, multiplicity, and choice within the eigenspace. The proof combines
exact unfolding, semiclassical defect measures, and the Veech dichotomy to
exclude concentration near the saddle/conic network and inside periodic
cylinders. We also obtain a reflection-overlap observability result for the
full rhombus and isolate the remaining saddle-network obstruction for arbitrary
open observation sets.
\end{abstract}


\end{frontmatter}

\section{Introduction}
\label{sec:introduction}

Let \(\Omega\) be a bounded domain and \(W\subset\Omega\) an open observation
set. We ask when there is a constant \(c_W>0\), independent of the eigenvalue
and of the choice of vector in a possibly multiple eigenspace, such that
\begin{equation}
    \|u\|_{L^2(W)}^2\ge c_W\|u\|_{L^2(\Omega)}^2
    \qquad\text{for every Laplace eigenfunction }u.
    \label{eq:intro-observability}
\end{equation}
This stationary observability estimate rules out exceptional high-frequency
sequences of exact eigenfunctions whose mass avoids \(W\); it is stronger than
qualitative unique continuation but weaker than equidistribution.
For rational polygons, the reflection unfolding is part of the problem. When
the boundary signs define a scalar character of the reflection group, as in
pure Dirichlet or Neumann sectors, the scalar spectral problem may close on a
torus or twisted torus. We consider a mixed regime in which this closure
fails. The reflection signs have nontrivial \(\mathbb Z_2\) monodromy, and a
scalar eigenfunction cannot close consistently on the usual toric unfolding.

Reflection across the short diagonal of the equilateral rhombus decomposes
the Dirichlet and Neumann problems into the DDD/NNN and DDN/NND sectors on
an equilateral triangle. The pure sectors admit scalar-character unfoldings.
In the mixed sectors, the reflection monodromy is nontrivial; it is
trivialized on a two-sheeted branched arithmetic translation surface of
genus two with two \(4\pi\) cone points. This cover has minimal possible
degree.
The directional flow on this surface satisfies the Veech dichotomy. We call
an open set admissible if its lift contains a fixed punctured neighborhood
of the cone set and meets every regular closed orbit in each completely
periodic direction. No corresponding condition is imposed in uniquely
ergodic directions. For every admissible set \(W\), the main theorem gives
\[
    \|u\|_{L^2(W)}^2
    \ge c_W\|u\|_{L^2(T)}^2
\]
for every DDN or NND eigenfunction \(u\), where \(c_W\) is independent of
the eigenvalue, its multiplicity, and the choice of \(u\) in the eigenspace.

The proof acts directly on exact eigensequences. A hypothetical dark sequence
lifts to a semiclassical defect measure on the branched surface. Away from the
cone set, propagation makes the measure invariant under directional flow. The
Veech dichotomy leaves two cases: unique ergodicity, which excludes a nonzero
dark invariant component, and complete periodicity, where invariant-measure
disintegration reduces the problem to maximal cylinders. The periodic-orbit
part of admissibility excludes the latter, while the punctured-cone condition
prevents residual mass from escaping through the saddle/conic network. This
is an all-eigenfunction statement, unlike density-one equidistribution on
rational polygons \cite{MarklofRudnick2012}; periodic-cylinder scarring at the
quasimode level \cite{FriedlandUeberschaer} also shows why the periodic
alternative cannot simply be ignored.

The mechanism differs from the toric observability available for pure
boundary conditions. Strong stationary and Schr\"odinger observability results
are known on flat tori \cite{BurqZworski2019,BurqGermainSorellaZhu2026}, and
Alphonse--Lafontaine obtain Schr\"odinger observability on the equilateral
triangle for pure Dirichlet and Neumann conditions by reduction to rational
twisted tori \cite{AlphonseLafontaine2025}. The DDN/NND sectors do not admit
that scalar toric reduction: non-character reflection monodromy changes the
phase space on which observability must be proved.

The admissible class is stable under enlargement and contains explicit open
families; in particular, complements of sufficiently small generic interior
holes are admissible. We also recombine the mixed and pure sectors to obtain a
reflection-overlap observability theorem on the full rhombus and a strict
fixed-frequency angle between the relevant parity restrictions. For arbitrary
nonempty open observation sets, the argument isolates one remaining issue:
exact-spectral transfer of mass through the saddle/conic network at
frequency-dependent scales. We record this reduction separately; it is not
used in the unconditional results.

\subsection*{Main results and proof architecture}

The paper has four main steps. First, we identify the DDN/NND sectors as
non-character reflection problems and construct their minimal-degree scalar
resolution on the branched genus-two translation surface. Second, a dark exact
eigensequence is reduced to an invariant semiclassical measure away from the
cone set. Third, the Veech dichotomy and invariant-measure disintegration
exclude mass in uniquely ergodic directions and in periodic cylinders under
the admissibility hypothesis. Fourth, the mixed-sector estimate is recombined
with the pure sectors to obtain the full-rhombus result. The same analysis
separates the two high-frequency obstructions, the saddle/conic network and
periodic cylinders, and reduces the arbitrary-open endpoint to the former.

The rest of the paper reviews the closest literature, introduces the geometry
and notation, states the main results, constructs the exact unfolding, proves
the observability theorem, returns to the full rhombus, and finally discusses
the arbitrary-open saddle-network obstruction.


\section{Related work}
\label{sec:related-work}

\paragraph{Tori and the equilateral triangle}
Burq--Zworski proved Schr\"odinger observability on two-dimensional tori for
rough localization functions \cite{BurqZworski2019}, and
Burq--Germain--Sorella--Zhu obtained stationary trace and observability
inequalities for toral Laplace eigenfunctions with respect to general Borel
measures \cite{BurqGermainSorellaZhu2026}. For the equilateral triangle,
Alphonse--Lafontaine proved Schr\"odinger observability for pure Dirichlet and
Neumann conditions by a tiling reduction to a rational twisted torus
\cite{AlphonseLafontaine2025}. These results cover the toric mechanism relevant
to the pure DDD/NNN sectors. The mixed DDN/NND sectors considered here do not
admit the same scalar torus realization because their reflection signs carry
nontrivial local \(\mathbb Z_2\) monodromy.

\paragraph{Mixed boundary conditions and polygonal concentration}
Mixed Dirichlet--Neumann spectra on triangles and reflection-group domains
have been studied, among others, by Siudeja \cite{Siudeja2016} and Stempak
\cite{Stempak2022}. The latter treats sign assignments induced by characters
of finite reflection groups; the DDN/NND patterns here lie outside that
framework. For polygonal concentration, Hassell and collaborators proved that
every neighborhood of the complete vertex set captures a uniform positive
fraction of every eigenfunction \cite{HassellHillairetMarzuola2009}, while
Ceki\'c et al. developed related control estimates involving periodic tubes
and nonsmooth boundary skeletons \cite{CekicGeorgievMukherjee2020}. In our
branched unfolding these two themes reappear as control near the conic set and
interception of regular periodic cylinders.

\paragraph{Quantum limits, translation surfaces, and cone points}
Marklof--Rudnick proved configuration-space equidistribution along a
density-one subsequence on rational polygons \cite{MarklofRudnick2012}; recent
work of Hippi--Mikkelsen strengthens this density-one perspective for a class
of rational polygons \cite{HippiMikkelsen2026}. Uniform observability must also
control exceptional exact eigensequences. At the quasimode level,
Friedland--Uebersch\"ar exhibit scarring on periodic structures of translation
surfaces \cite{FriedlandUeberschaer}. The conic points created by our unfolding
also connect the problem with microlocal propagation on singular spaces; see
Yang \cite{Yang2022ProductCones} and Hintz \cite{Hintz2024Cone}. Our
punctured-cone hypothesis removes limiting dark mass from this singular
region, after which the argument uses the regular directional dynamics.

\paragraph{Position of the present result}
The equilateral rhombus is the explicit model, but the structural distinction
used here is between reflection data that close to a scalar character and
reflection data with nontrivial scalar monodromy. In the former case a toric
or twisted-toric realization may be available; in the latter the scalar phase
space must first be changed. The DDN/NND sectors give an explicit instance in
which this produces a higher-genus translation surface and forces the
observability argument to combine singular-set control with directional
dynamics. We do not claim a theorem for arbitrary non-character reflection
systems. Rather, the present model is one in which the exact unfolding and the
full high-frequency reduction can both be carried out. The extension from
admissible sets to arbitrary nonempty open observations is reduced separately
to a quantitative saddle-network problem.


\section{Geometry and mixed reflection monodromy}
\label{sec:geometry}
Let \(\Rh\) be the equilateral \(60^\circ\)-\(120^\circ\) rhombus and let
\(\M\) denote its short diagonal.  Reflection across \(\M\) is denoted by
\(\sigma\).  Cutting \(\Rh\) along \(\M\) produces two congruent equilateral
triangles; throughout the paper \(\T\) denotes one of them and \(\T^\circ\) its interior.  The two sides of
\(\partial\T\) inherited from \(\partial\Rh\) are called the \emph{outer}
sides, while the third side is \(\M\).  The geometry and the parity-induced condition on \(\M\) are summarized in \cref{fig:parity-reduction}.

\begin{figure}[t]
\centering
\begin{tikzpicture}[
    scale=0.92,
    >=Latex,
    font=\small,
    every node/.style={inner sep=1.5pt}
]

\coordinate (A) at (-6.0,0.5);
\coordinate (B) at (-4.2,2.25);
\coordinate (C) at (-2.4,0.5);
\coordinate (D) at (-4.2,-1.25);

\draw[line width=1.1pt]
    (A)--(B)--(C)--(D)--cycle;

\draw[line width=0.95pt,densely dashed]
    (B)--(D);

\node[left]  at (A) {$A$};
\node[above] at (B) {$B$};
\node[right] at (C) {$C$};
\node[below] at (D) {$D$};

\node[left] at (-4.35,0.55) {$M$};

\node[align=center] at (-4.2,3.05)
{
outer boundary condition\\
$\mathsf B\in\{D,N\}$
};

\node[align=center] at (-4.2,-2.05)
{
$60^\circ$--$120^\circ$ rhombus $R$
};

\draw[->,line width=0.9pt]
    (-2.10,0.5) -- (-0.80,0.5)
    node[midway,above=3pt,align=center]
    {restrict to\\one half};

\coordinate (T1) at (-0.20,-0.45);
\coordinate (T2) at (1.25,2.25);
\coordinate (T3) at (2.70,-0.45);

\draw[line width=1.1pt]
    (T1)--(T2)--(T3);

\draw[line width=0.95pt,densely dashed]
    (T3)--(T1);

\node[below left]  at (T1) {$Q$};
\node[above]       at (T2) {$P$};
\node[below right] at (T3) {$R$};

\node[left]  at (0.45,1.00) {$\mathsf B$};
\node[right] at (2.05,1.00) {$\mathsf B$};

\node[below] at (1.25,-0.50) {$M$};

\node[align=center] at (1.25,-1.45)
{
equilateral triangle $T$
};

\coordinate (Branch) at (3.55,0.50);
\fill (Branch) circle (2pt);

\node[align=center,right=5pt of Branch]
{
parity\\
across $M$
};

\draw[line width=0.8pt]
    (2.85,0.50)--(Branch);

\coordinate (E1) at (5.35,2.10);
\coordinate (E2) at (6.35,3.30);
\coordinate (E3) at (7.35,2.10);

\draw[->,line width=0.9pt]
    (Branch) to[out=40,in=195] (5.00,2.30);

\draw[line width=1.05pt]
    (E1)--(E2)--(E3);

\draw[line width=0.95pt,densely dashed]
    (E3)--(E1);

\node[below left]  at (E1) {$Q$};
\node[above]       at (E2) {$P$};
\node[below right] at (E3) {$R$};

\node[left]  at (5.72,2.66) {$\mathsf B$};
\node[right] at (6.98,2.66) {$\mathsf B$};

\node[above=9pt] at (E2)
    {$\sigma$-even};

\node[below=2pt] at (6.35,2.05)
    {$M$};

\node[below=20pt,align=center] at (6.35,2.10)
{
\textbf{Neumann on $M$}\\[-1pt]
$\partial_\nu u|_M=0$
};

\coordinate (O1) at (5.35,-2.10);
\coordinate (O2) at (6.35,-0.90);
\coordinate (O3) at (7.35,-2.10);

\draw[->,line width=0.9pt]
    (Branch) to[out=-40,in=165] (5.00,-1.90);

\draw[line width=1.05pt]
    (O1)--(O2)--(O3);

\draw[line width=1.05pt]
    (O3)--(O1);

\node[below left]  at (O1) {$Q$};
\node[above]       at (O2) {$P$};
\node[below right] at (O3) {$R$};

\node[left]  at (5.72,-1.54) {$\mathsf B$};
\node[right] at (6.98,-1.54) {$\mathsf B$};

\node[above=9pt] at (O2)
    {$\sigma$-odd};

\node[below=2pt] at (6.35,-2.15)
    {$M$};

\node[below=20pt,align=center] at (6.35,-2.10)
{
\textbf{Dirichlet on $M$}\\[-1pt]
$u|_M=0$
};

\end{tikzpicture}

\caption{
Parity reduction across the short diagonal.
The two outer sides of the half-triangle inherit the rhombus boundary
condition $\mathsf B$, while parity across $M$ determines the boundary
condition on $M$: even parity gives Neumann and odd parity gives Dirichlet.
Thus Dirichlet rhombus data split into DDN/DDD sectors, whereas Neumann
rhombus data split into NNN/NND sectors.
}
\label{fig:parity-reduction}
\end{figure}
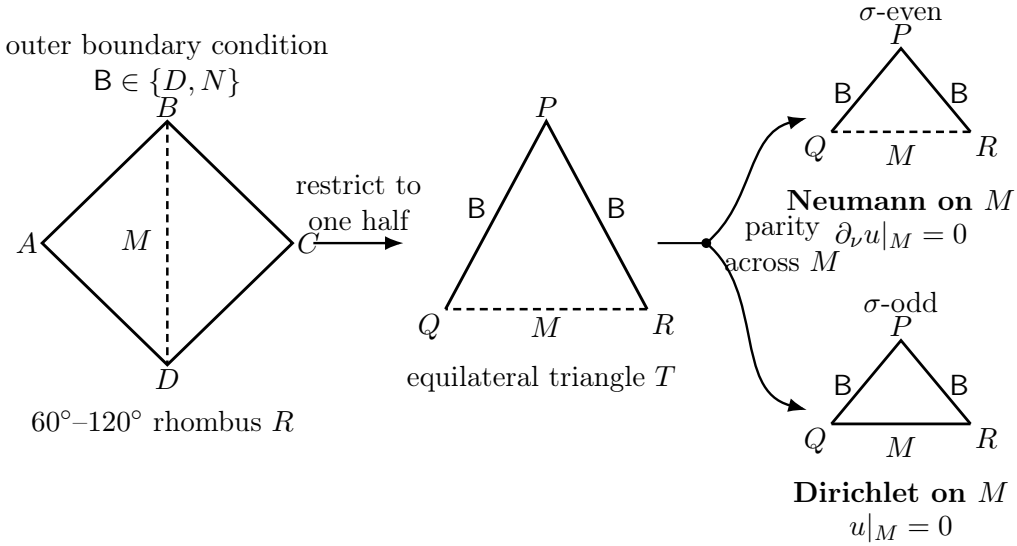

\begin{proposition}[Parity decomposition]
\label{prop:parity-decomposition}
Let \(H_{\Rh,D}\) and \(H_{\Rh,N}\) be the Dirichlet and Neumann Laplacians on
\(\Rh\).  Reflection across \(\M\) commutes with both operators.  Hence, every
rhombus eigenspace decomposes orthogonally into the \(\pm1\) eigenspaces of
\(\sigma\).\\
After restriction to \(\T\), followed by multiplication by \(\sqrt2\), these
parity subspaces are unitarily identified as follows:
\[
 \begin{array}{c|cc}
 & \sigma\text{-even} & \sigma\text{-odd}\\ \hline
 H_{\Rh,D} & H_{DDN} & H_{DDD}\\
 H_{\Rh,N} & H_{NNN} & H_{NND}.
 \end{array}
\]
Here, the third boundary-condition letter is the condition on \(\M\).  In
particular, after these canonical identifications,
\[
 E_\lambda^{\Rh,D}\cong E_\lambda^{DDD}\oplus E_\lambda^{DDN},
 \qquad
 E_\lambda^{\Rh,N}\cong E_\lambda^{NNN}\oplus E_\lambda^{NND}.
\]
The identifications preserve eigenvalues and multiplicities.
\end{proposition}

\begin{proof}
The reflection \(\sigma\) is an isometry of \(\Rh\) preserving
\(\partial\Rh\), and it preserves both the Dirichlet form domain
\(H_0^1(\Rh)\) and the Neumann form domain \(H^1(\Rh)\).  Consequently the
associated self-adjoint Laplacian commutes with the unitary involution
\(U_\sigma f=f\circ\sigma\).  The spectral theorem for the bounded
self-adjoint involution \(U_\sigma\) therefore gives an orthogonal splitting
of every Laplace eigenspace into its even and odd parity parts.

Let \(u\) belong to one of these parity subspaces and put \(v=u|_\T\).  On the
two outer sides, \(v\) inherits the boundary condition imposed on
\(\partial\Rh\).  On the relative interior of \(\M\), evenness gives
\(\partial_\nu v=0\): in local coordinates \((s,n)\) with \(\sigma(s,n)=(s,-n)\),
an even solution satisfies \(\partial_n u(s,0)=0\).  Oddness gives
\(v|_\M=0\). Thus, the four cases in the table follow.

Conversely, a triangle eigenfunction in any of the displayed sectors extends
across \(\M\) by the indicated even or odd reflection.  The reflected pieces
have matching weak Cauchy data across \(\M\), so the extension lies in the
appropriate rhombus form domain and satisfies the rhombus eigenvalue equation
weakly; elliptic regularity then gives the corresponding rhombus
eigenfunction.  Restriction and reflection are inverse to one another.
Finally, because the two halves of \(\Rh\) are exchanged by \(\sigma\) and
\(|u\circ\sigma|=|u|\),
\[
 \|u\|_{L^2(\Rh)}^2=2\|u|_\T\|_{L^2(\T)}^2.
\]
Hence, \(u\mapsto\sqrt2\,u|_\T\) is unitary on each parity subspace, proving
the stated identifications and multiplicity preservation.
\end{proof}

The four resulting triangle problems, with the convention that the third letter denotes the condition on \(\M\), are displayed in \cref{fig:four-sectors}.

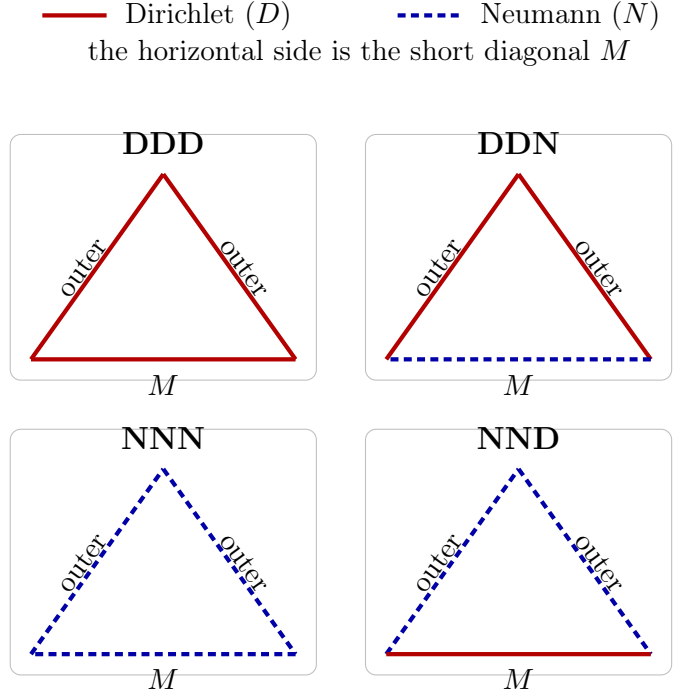
\begin{figure}[t]
\centering
\begin{tikzpicture}[x=1cm,y=1cm,font=\small]
  \tikzset{
    Dedge/.style={line width=1.6pt,draw=red!70!black},
    Nedge/.style={line width=1.6pt,draw=blue!65!black,densely dashed},
    panel/.style={draw=black!25,rounded corners,inner sep=5pt}
  }
  \draw[Dedge] (0.15,4.55)--(1.00,4.55); \node[right] at (1.10,4.55) {Dirichlet ($D$)};
  \draw[Nedge] (4.85,4.55)--(5.70,4.55); \node[right] at (5.80,4.55) {Neumann ($N$)};
  \node at (4.35,4.05) {the horizontal side is the short diagonal $M$};

  \newcommand{\sectorpanel}[6]{%
    \begin{scope}[shift={(#1,#2)}]
      \node[panel,minimum width=4.05cm,minimum height=3.25cm] at (1.75,1.35) {};
      \coordinate (a) at (0,0); \coordinate (b) at (1.75,2.45); \coordinate (c) at (3.5,0);
      \draw[#3] (a)--(b);
      \draw[#4] (b)--(c);
      \draw[#5] (c)--(a);
      \node[font=\bfseries] at (1.75,2.86) {#6};
      \node[rotate=54] at (0.67,1.22) {outer};
      \node[rotate=-54] at (2.83,1.22) {outer};
      \node[below] at (1.75,-0.06) {$M$};
    \end{scope}
  }
  \sectorpanel{0}{0}{Dedge}{Dedge}{Dedge}{DDD}
  \sectorpanel{4.7}{0}{Dedge}{Dedge}{Nedge}{DDN}
  \sectorpanel{0}{-3.9}{Nedge}{Nedge}{Nedge}{NNN}
  \sectorpanel{4.7}{-3.9}{Nedge}{Nedge}{Dedge}{NND}

  \node[align=center,font=\small] at (2.0,-4.65) {Dirichlet rhombus: odd $\to$ DDD, even $\to$ DDN};
  \node[align=center,font=\small] at (7.0,-5.08) {Neumann rhombus: even $\to$ NNN, odd $\to$ NND};
\end{tikzpicture}
\caption{The four triangle sectors produced by parity reduction.  The hard sectors are DDN and NND: the two outer sides have the same boundary condition, whereas the short diagonal $M$ carries the opposite one.  Solid red edges denote Dirichlet conditions and dashed blue edges denote Neumann conditions; the line styles remain distinguishable in grayscale.}
\label{fig:four-sectors}
\end{figure}


\begin{definition}[Observation functional]
For nonempty open \(W\subset\T\) and nonzero \(u\), set
\[
 \mathcal O_W(u)=\frac{\norm{u}_{L^2(W)}^2}{\norm{u}_{L^2(\T)}^2}.
\]
\end{definition}


\section{Notation and proof conventions}
\label{sec:notation}

This section fixes the notation used from the statement of the main
results through the appendices.  Symbols are introduced here before
their first technical use.

\paragraph{Domains and symmetry}
The rhombus is denoted by \(\Rh\), its short diagonal by \(\M\), and
one equilateral half by \(\T\).  Reflection across \(\M\) is
\(\sigma\).  Observation sets are denoted by \(W\) on \(\T\) and,
when no confusion can arise, also by \(W\) on \(\Rh\).

\paragraph{Operators and eigenspaces}
For a boundary-condition label
\[
 B\in\{DDD,DDN,NNN,NND\},
\]
\(H_B\) denotes the corresponding self-adjoint Laplacian on \(\T\),
and
\[
E_\lambda^B:=\ker(H_B-\lambda).
\]
In the high-frequency argument, we write \(\lambda=h^{-2}\) and
\[
P_h:=-h^2\Delta-1.
\]
Thus, an exact eigenfunction at energy \(\lambda\) satisfies
\(P_hu_h=0\).  The phrase \emph{exact eigenfunction sequence} always
means a sequence satisfying this equation exactly, not merely up to a
quasimode error.

\paragraph{Unfolded surface and phase space}
The finite flat/conic unfolding is denoted by \(X\), its finite cone
set by \(\mathcal C_X\), and its regular part by
\[X^\circ:=X\setminus\mathcal C_X.\]
The unfolding carries a canonical continuous folding map
\(F:X\to\T\), obtained by sending each triangle copy back to \(\T\) by
its defining reflection isometry; the copy maps agree on their common seams.
For a relatively open observation set \(W\subset\T\), define its full
regular lift by
\[
  \widetilde W:=F^{-1}(W)\cap X^\circ.
\]
Thus, \(\widetilde W\) is open in \(X^\circ\).  Away from the unfolded
seams it is equivalently the union of the finitely many triangle-copy lifts
\(\iota_\alpha(W\cap\T^\circ)\).  The seams and cone points have area
zero, so this convention gives the same \(L^2\)-mass identities as the
copywise definition while making set-theoretic conditions such as the
punctured cone guard literal.
The cotangent projection on the regular part is \(\pi:T^*X^\circ\to X^\circ\).  We use
\[
p(x,\xi)=|\xi|^2-1,\qquad
S^*X^\circ:=\{(x,\xi)\in T^*X^\circ:p(x,\xi)=0\},
\]
and write \(\Phi_t\) for the Hamiltonian flow of \(p\) wherever the
regular flow is defined.  A semiclassical measure generated by an
\(L^2\)-bounded sequence is denoted by \(\mu\).

\paragraph{Microlocal mass}
If \(\mathcal Y\subset S^*X^\circ\) is a fixed microlocal component, choose
\(\chi_{\mathcal Y}\in C_c^\infty(T^*X^\circ)\) equal to one on a slightly
smaller representative of \(\mathcal Y\).  We write
\[
M_{\mathcal Y,h}(u):=
\langle \Op_h(\chi_{\mathcal Y})u,u\rangle_{L^2(X)}.
\]
Statements about ``mass in a directional component'' refer to this
microlocal quantity (or to its limiting measure), not to an
\(L^2\)-norm over a subset of phase space.

\paragraph{Uniformity convention}
Unless explicitly stated otherwise, constants in the high-frequency
argument may depend on the fixed observation set and the fixed finite
propagation construction, but not on \(h\) or on the chosen normalized
eigenfunction.


\section{Main results}
\label{sec:main-results}

We write \(H_{DDN}\) and \(H_{NND}\) for the self-adjoint Laplacians on
\(\T\) with the indicated mixed boundary conditions, with the third condition
imposed on \(\M\). Let \(X\) be the exact mixed-sector unfolding from
\cref{prop:unfolding}; for an open set \(W\subset\T\), its full lift in the
regular part is denoted by
\(\widetilde W\subset X^\circ:=X\setminus\mathcal C_X\).

\begin{definition}[Admissible observation geometry]
\label{def:admissible-observation}
A nonempty open set \(O\subset X^\circ\) belongs to
\(\mathcal O_{\rm adm}(X)\) if both conditions below hold.
\begin{enumerate}[label=\textnormal{(A\arabic*)},leftmargin=2.6em]
\item \emph{Punctured cone guard.} There exists
\(\varepsilon_{\mathcal C}>0\) such that
\[
 U_{\mathcal C}^\times
 :=\{x\in X^\circ:0<d_X(x,\mathcal C_X)<\varepsilon_{\mathcal C}\}
 \subset O.
\]
\item \emph{Periodic-cylinder interception.} For every completely periodic
direction \(\theta\) and every maximal cylinder \(C_{\theta,\alpha}\) in that
direction, every regular closed orbit contained in \(C_{\theta,\alpha}\)
meets \(O\).
\end{enumerate}
For the triangle, set
\[
 \mathcal O_{\rm adm}(\T)
 :=\{W\subset\T:\ W\text{ is nonempty and open, and }
                 \widetilde W\in\mathcal O_{\rm adm}(X)\}.
\]
\end{definition}

\begin{remark}[Admissibility versus geometric control]
\label{rem:admissibility-not-gcc}
The admissibility condition is not a geometric control condition in disguise.
It imposes no hitting requirement, and in particular no uniform hitting time,
on trajectories in uniquely ergodic directions. Geometric interception is
required only in completely periodic directions, precisely because these are
the directions in which invariant probability measures may remain supported
on proper regular subsets of the unfolded surface. Thus \textnormal{(A2)} is
an invariant-measure exclusion condition rather than a trajectorywise control
condition. Condition \textnormal{(A1)} is logically separate: it removes the
singular network from the regular propagation problem. This criterion is
specific to the stationary eigenfunction problem considered here and is not
intended as a geometric control condition for time-dependent wave propagation.
\end{remark}

\begin{theorem}[Mixed-sector uniform observability]
\label{thm:mixed-uniform}
Let \(B\in\{DDN,NND\}\) and \(W\in\mathcal O_{\rm adm}(\T)\). Then, there
exists \(c_W>0\) such that, for every eigenvalue \(\lambda\) of \(H_B\) and
every \(u\in E_\lambda^B\),
\[
 \norm{u}_{L^2(W)}^2\ge c_W\norm{u}_{L^2(\T)}^2.
\]
The constant is independent of the eigenvalue, the eigenspace dimension, and
the choice of eigenfunction.
\end{theorem}

\begin{corollary}[Concrete small-hole observations]
\label{cor:small-hole-admissible}
There exists a dense open set of regular interior points \(x_0\in\T\) with the
following property: for all sufficiently small \(r>0\) such that
\(\overline{B_\T(x_0,r)}\Subset\T\), the observation region
\[
 W_r:=\T\setminus\overline{B_\T(x_0,r)}
\]
belongs to \(\mathcal O_{\rm adm}(\T)\).  Consequently, for each
\(B\in\{DDN,NND\}\) there is \(c_{W_r}>0\) such that
\[
 \norm{u}_{L^2(W_r)}^2\ge c_{W_r}\norm{u}_{L^2(\T)}^2
 \qquad\text{for every }u\in E_\lambda^B\text{ and every }\lambda.
\]
The same conclusion holds for a finite family of sufficiently small mutually separated
interior holes whose distinct unfolded lifts remain separated.
\end{corollary}

\begin{theorem}[Reflection-overlap observability on the rhombus]
\label{thm:reflection-overlap}
Let \(W\subset\Rh\) be nonempty and open and set
\[
 V:=\operatorname{int}(W\cap\sigma(W)).
\]
Assume \(V\neq\varnothing\) and, after identifying either half of the rhombus
with \(\T\), that \(V\cap\T\in\mathcal O_{\rm adm}(\T)\). Then, there exists
\(c_W>0\) such that every Dirichlet rhombus eigenfunction satisfies
\[
 \norm{u}_{L^2(W)}^2\ge c_W\norm{u}_{L^2(\Rh)}^2.
\]
The analogous statement holds for the Neumann rhombus.
\end{theorem}

\begin{proposition}[Strict fixed-frequency cross-parity angle]
\label{prop:fixed-angle}
Let \(W\subset\T\) be nonempty and open.  Fix one of the two parity pairs
\[
 (B_+,B_-)\in\{(DDN,DDD),(NNN,NND)\},
\]
and let \(\lambda\) be a common eigenvalue of \(H_{B_+}\) and \(H_{B_-}\).
Define
\[
 \rho_W^{B_+,B_-}(\lambda)
 =\sup_{\substack{0\ne u_+\in E_\lambda^{B_+}\\
                  0\ne u_-\in E_\lambda^{B_-}}}
 \frac{|\langle u_+,u_-\rangle_{L^2(W)}|}
 {\norm{u_+}_{L^2(W)}\norm{u_-}_{L^2(W)}} .
\]
Then,
\[
 \rho_W^{B_+,B_-}(\lambda)<1.
\]
\end{proposition}

\begin{remark}[Role of the cone condition]
\label{rem:cone-guard-status}
The cone-neighborhood condition in \textnormal{(A1)} ensures that trajectories
reaching a cone point first enter the observation region. Consequently, the
high-frequency argument can be carried out entirely by propagation on the
regular part of the unfolded surface, without requiring propagation through
the conic singularities. The extension to arbitrary nonempty open observation
sets, where this mechanism is no longer available, is considered separately
in \cref{sec:endpoint}. In particular, the proof of
\cref{thm:mixed-uniform} relies only on regular propagation together with the
periodic-cylinder and unique-ergodicity arguments developed below.
\end{remark}

\section{Exact flat-surface realization of the mixed sectors}
\label{sec:flat-surface}
Let
\[
 e_1=(1,0),\qquad
 e_2=\left(\frac12,\frac{\sqrt3}{2}\right),\qquad
 \Lambda_\triangle=\mathbb Z e_1\oplus\mathbb Z e_2.
\]
The distinction between the pure and mixed parity sectors is already visible
at the level of the reflection signs. We begin by making this distinction
precise.

\begin{proposition}[Mixed Coxeter monodromy and minimal-degree scalar resolution]
\label{prop:mixed-monodromy}
Let \(s_1,s_2,s_3\) denote the reflections in the sides of an equilateral
triangle. They satisfy the Coxeter relations
\[
    s_i^2=1,
    \qquad
    (s_i s_j)^3=1
    \quad (i\ne j).
\]
The reflection signs associated with the DDN and NND boundary conditions are,
respectively,
\[
    (-1,-1,+1)
    \qquad\text{and}\qquad
    (+1,+1,-1).
\]
Neither assignment extends to a scalar character of the equilateral
reflection group. Moreover, a positively oriented loop around each mixed
D/N vertex has nontrivial \(\mathbb Z_2\)-monodromy. This monodromy is
trivialized by a two-sheeted translation cover of the triangular torus,
branched precisely over the two mixed vertices. Since the local monodromy is
the nontrivial element of \(\mathbb Z_2\), any scalar cover that trivializes
it has degree at least two; hence this resolution has minimal degree.
The resulting compact translation surface \(X\) has exactly two conic
singularities, each of angle \(4\pi\), and has genus two.
\end{proposition}

\begin{proof}
Suppose first that the reflection signs extend to a scalar character
\(\chi\), and write \(\varepsilon_i=\chi(s_i)\in\{\pm1\}\).
Applying \(\chi\) to the Coxeter relation
\((s_i s_j)^3=1\) gives
\[
    (\varepsilon_i\varepsilon_j)^3=1.
\]
At a mixed D/N vertex, however, the two adjacent sides carry opposite
reflection signs, so that
\[
    \varepsilon_i\varepsilon_j=-1,
    \qquad
    (\varepsilon_i\varepsilon_j)^3=-1,
\]
a contradiction. Hence, neither the DDN nor the NND sign assignment extends
to a scalar character of the equilateral reflection group.
The same obstruction has a local geometric interpretation. At a mixed
vertex of angle \(\pi/3\), the composition of the two adjacent reflections
is a rotation through \(2\pi/3\). Three such compositions therefore make
one full turn around the vertex, while the reflected solution acquires the
factor
\[
    (\varepsilon_i\varepsilon_j)^3=-1.
\]
Thus, a positively oriented loop around each mixed vertex has nontrivial
\(\mathbb Z_2\)-monodromy. Since this monodromy is detected by a loop around
a contractible puncture, it cannot be removed merely by replacing the
translation lattice by a finite-index sublattice.

The monodromy is trivialized by passing to the two-sheeted cover in which
continuation around either mixed vertex exchanges the two sheets. Since the
local monodromy is the nontrivial element of \(\mathbb Z_2\), any scalar
cover trivializing it has degree at least two; hence this construction has
minimal possible degree. The cover is branched precisely over the two mixed
vertices. Each is a simple branch point over a regular point of the
triangular torus, so the local degree-two pullback changes the total angle
from \(2\pi\) to \(4\pi\).

Let
\[
    \Pi:X\longrightarrow\mathbb T^2_\triangle
\]
denote the resulting degree-two branched cover. Since the base has genus one
and there are exactly two simple branch points, the Riemann--Hurwitz formula
gives
\[
    2g(X)-2
    =
    2\bigl(2g(\mathbb T^2_\triangle)-2\bigr)+2
    =
    2,
\]
and therefore
\[
    g(X)=2.
\]

Finally, the DDN and NND reflection signs differ by multiplication by
\(-1\) on each generating reflection. Every orientation-preserving word has
even reflection length, so this overall sign cancels on the
orientation-preserving subgroup. Consequently, the two mixed sectors
determine the same underlying orientation-preserving branched translation
surface.
\end{proof}

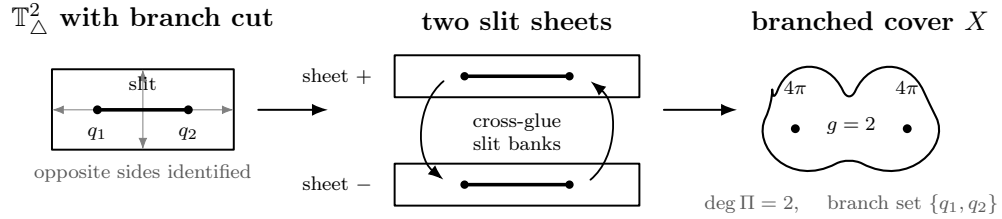
\begin{figure}[H]
\centering
\resizebox{0.98\textwidth}{!}{%
\begin{tikzpicture}[
    font=\scriptsize,
    >=Latex,
    line cap=round,
    line join=round
]


\node[font=\small\bfseries] at (-5.35,1.75)
  {$\mathbb T^2_\triangle$ with branch cut};

\draw[line width=.8pt]
  (-6.65,-.10) rectangle (-4.05,1.05);

\draw[<->,black!50]
  (-6.65,.47)--(-4.05,.47);

\draw[<->,black!50]
  (-5.35,-.10)--(-5.35,1.05);

\fill (-6.00,.47) circle (1.8pt)
  node[below=3pt] {$q_1$};

\fill (-4.70,.47) circle (1.8pt)
  node[below=3pt] {$q_2$};

\draw[line width=1.45pt]
  (-5.94,.47)--(-4.76,.47);

\node[above=4pt] at (-5.35,.47) {slit};

\node[black!65] at (-5.35,-.43)
  {opposite sides identified};


\draw[->,line width=.8pt]
  (-3.70,.47)--(-2.70,.47);


\node[font=\small\bfseries] at (0,1.75)
  {two slit sheets};


\draw[line width=.8pt]
  (-1.75,.65) rectangle (1.75,1.25);

\node[left=5pt] at (-1.75,.95)
  {sheet $+$};

\fill (-.75,.95) circle (1.7pt);
\fill ( .75,.95) circle (1.7pt);

\draw[line width=1.45pt]
  (-.69,.95)--(.69,.95);


\draw[line width=.8pt]
  (-1.75,-.90) rectangle (1.75,-.30);

\node[left=5pt] at (-1.75,-.60)
  {sheet $-$};

\fill (-.75,-.60) circle (1.7pt);
\fill ( .75,-.60) circle (1.7pt);

\draw[line width=1.45pt]
  (-.69,-.60)--(.69,-.60);

%

\draw[->,line width=.75pt]
  (-1.05,.88)
    .. controls (-1.45,.55) and (-1.45,-.20) ..
  (-1.05,-.53);

\draw[->,line width=.75pt]
  (1.05,-.53)
    .. controls (1.45,-.20) and (1.45,.55) ..
  (1.05,.88);

\node[align=center,fill=white,inner sep=2pt]
  at (0,.16)
  {cross-glue\\slit banks};


\draw[->,line width=.8pt]
  (2.10,.47)--(3.15,.47);


\node[font=\small\bfseries] at (5.05,1.75)
  {branched cover $X$};

\draw[line width=.85pt,rounded corners=6pt]
  (3.65,.55)
    .. controls (3.90,1.20) and (4.45,1.20) .. (4.75,.55)
    .. controls (5.05,1.20) and (5.60,1.20) .. (5.90,.55)
    .. controls (6.20,-.05) and (5.60,-.65) .. (4.75,-.20)
    .. controls (3.90,-.65) and (3.35,-.05) .. (3.65,.55)
    -- cycle;

\fill (3.98,.20) circle (1.8pt);
\fill (5.58,.20) circle (1.8pt);

\node[above=10pt] at (3.98,.20) {$4\pi$};
\node[above=10pt] at (5.58,.20) {$4\pi$};

\node at (4.78,.25) {$g=2$};

\node[align=center,black!65]
  at (4.78,-.88)
  {$\deg\Pi=2$, \quad branch set $\{q_1,q_2\}$};

\end{tikzpicture}
}

\caption{
Topological construction of the scalar mixed-sector unfolding.
The nontrivial $\mathbb Z_2$ monodromy is resolved by taking two copies of the
triangular torus slit between the mixed-vertex images $q_1,q_2$ and
cross-gluing the banks. The resulting degree-two cover is branched exactly at
$q_1,q_2$; each branch point becomes a $4\pi$ cone and
Riemann--Hurwitz gives genus two. The diagram records the covering and gluing
data, not a Euclidean fundamental polygon for $X$.
}
\label{fig:branched-cover-gluing}
\end{figure}

\begin{proposition}[Exact mixed-boundary unfolding]
\label{prop:unfolding}
Let \(u_h\) be an exact eigenfunction in either mixed sector, satisfying
\[
    (-h^2\Delta-1)u_h=0
    \qquad \text{on } \T .
\]
Then, the reflection construction of \cref{prop:mixed-monodromy} lifts \(u_h\)
to a function \(\widetilde u_h\) in the operator domain of the Friedrichs
Laplacian on \(X\), satisfying
\[
    (-h^2\Delta_X-1)\widetilde u_h=0
\]
in the global weak sense on \(X\). On each regular triangle copy,
\(\widetilde u_h\) is the pullback of \(u_h\) by an isometry, up to sign.
If \(N_{\rm sh}\) denotes the fixed number of triangle copies in the
unfolding, then
\[
    \|\widetilde u_h\|_{L^2(X)}^2
    =
    N_{\rm sh}\|u_h\|_{L^2(\T)}^2,
    \qquad
    \|\widetilde u_h\|_{L^2(\widetilde W)}^2
    =
    N_{\rm sh}\|u_h\|_{L^2(W)}^2.
\]
\end{proposition}

\begin{proof}
Across a Dirichlet side we extend \(u_h\) oddly, and across a Neumann side
evenly. The corresponding traces and normal derivatives match across each
regular unfolded seam in the weak Green identity, so no distributional
source is produced there. Iterating these reflections and resolving the
mixed monodromy by the two-sheeted construction of
\cref{prop:mixed-monodromy} yields a globally defined function
\(\widetilde u_h\) on \(X^\circ\) that is locally \(H^1\) and satisfies
\[
    (-h^2\Delta_X-1)\widetilde u_h=0
\]
weakly across every regular seam.

It remains to verify the behavior at the two cone points. The mixed
Laplacian on \(\T\) is the self-adjoint operator associated with the closed
mixed Dirichlet--Neumann quadratic form, and finite odd/even reflection
preserves the local \(H^1\) energy. Hence, \(\widetilde u_h\) belongs locally
to the Friedrichs form domain near each cone point. By the mixed-corner
analysis in \cref{app:conic}, the admissible angular orders are
\[
    \nu_k=3k+\frac32,
    \qquad k=0,1,2,\ldots .
\]
The Friedrichs condition selects the regular radial branches
\(r^{\nu_k}\); in particular, \(\nu_k\ge 3/2\). Consequently,
\(\widetilde u_h\) and its first derivatives are locally square integrable,
and the boundary contribution in Green's identity on a circle
\(r=\varepsilon\) tends to zero as \(\varepsilon\downarrow0\).
The weak Helmholtz identity therefore extends across each cone point.

Thus, \(\widetilde u_h\) belongs to the Friedrichs form domain and satisfies
the weak eigenvalue identity against every test function in that form
domain. By the representation theorem for the closed Friedrichs form,
\(\widetilde u_h\) belongs to the operator domain of the Friedrichs
Laplacian and is an exact eigenfunction of \(-\Delta_X\).

Finally, the unfolding consists of \(N_{\rm sh}\) congruent copies of
\(\T\). In each copy, the construction is an isometry followed, possibly,
by multiplication by \(-1\), and hence preserves the pointwise modulus.
Summing over the copies gives
\[
    \|\widetilde u_h\|_{L^2(X)}^2
    =
    N_{\rm sh}\|u_h\|_{L^2(\T)}^2.
\]
Since \(\widetilde W\) is the full lift of \(W\), the same argument restricted
to the lifted observation set gives
\[
    \|\widetilde u_h\|_{L^2(\widetilde W)}^2
    =
    N_{\rm sh}\|u_h\|_{L^2(W)}^2.
\]
\end{proof}

\begin{proposition}[Arithmetic Veech property of the mixed unfolding]
\label{prop:arithmetic-veech}
The translation surface \(X\) is affinely equivalent to a square-tiled
surface.  In particular, its Veech group is a lattice in
\(\mathrm{SL}(2,\mathbb R)\), and every direction on \(X\) is either
completely periodic or uniquely ergodic.
\end{proposition}

\begin{proof}
Choose one developed singular vertex as origin.  The surface \(X\) is
assembled from finitely many reflected equilateral-triangle copies, with all
edge identifications given by translations after the orientation-preserving
unfolding.  Hence, every developed vertex lies in the triangular lattice
\(\Lambda_\triangle\).  In particular, a relative period joining two cone
points is the displacement between two such developed vertices and therefore
belongs to \(\Lambda_\triangle\).  Likewise, a closed translation path lifts
in the development to a segment whose endpoints are two developed lifts of
the same point; composing with the triangle-copy identifications shows that
its displacement is again a vector of \(\Lambda_\triangle\).  Thus, all
absolute periods, as well as all relative periods between singularities, lie
in \(\Lambda_\triangle\).

Choose \(A\in\mathrm{GL}^+(2,\mathbb R)\) with
\(Ae_1=(1,0)\) and \(Ae_2=(0,1)\).  Then all absolute and singular-to-singular
relative periods of \(AX\) lie in \(\mathbb Z^2\).  Consequently, the
developing coordinate modulo \(\mathbb Z^2\) is independent of the chosen
path and defines a translation map
\[
 AX\longrightarrow\mathbb R^2/\mathbb Z^2.
\]
The relative-period statement implies that the two cone points have the same
image modulo \(\mathbb Z^2\); after translating the torus coordinate we may
take this image to be the origin.  Away from the cone points the map is a
local translation, hence a local homeomorphism.  Since \(AX\) is compact, the
map is therefore a finite translation covering, with branching possible only
at the cone points and hence only over the origin of the square torus.
Equivalently, \(AX\) is square-tiled.  The theorem of Gutkin--Judge then makes
the Veech group arithmetic and hence a lattice \cite{GutkinJudge2000}.
Veech's dichotomy for lattice surfaces gives the periodic/uniquely-ergodic
alternative \cite{Veech1989,SmillieWeiss2008}.
\end{proof}

\begin{remark}
The arithmetic structure enters the argument through the directional
dichotomy, which excludes proper minimal components from the support of a
semiclassical measure carrying no mass in the observation region. For
admissible observation sets, the fixed punctured neighborhood of the cone set
removes concentration near the conic points before the regular-flow analysis
is applied. Consequently, the proof of the observability theorem proceeds
entirely through propagation on the regular part and the resulting
periodic-cylinder decomposition.
\end{remark}


\section{High-frequency reduction away from the cone set}
\label{sec:high-frequency}

Throughout this section, \(X\) denotes the exact Friedrichs
translation-surface unfolding constructed in \cref{prop:unfolding}. Let
\(W\in\mathcal O_{\rm adm}(\T)\) and set \(O=\widetilde W\).

Suppose, for contradiction, that the high-frequency observability estimate
fails. Then, there exist normalized exact mixed-sector eigenfunctions \(u_j\)
and a sequence \(h_j\to0\) such that
\[
    (-h_j^2\Delta-1)u_j=0,\qquad
    \|u_j\|_{L^2(\T)}=1,\qquad
    \|u_j\|_{L^2(W)}\longrightarrow0.
\tag{7.1}
\]
After unfolding and normalization by \(N_{\rm sh}^{-1/2}\),
\cref{prop:unfolding} gives exact Friedrichs eigenfunctions
\(\widetilde u_j\) on \(X\) satisfying
\[
    \|\widetilde u_j\|_{L^2(X)}=1,
    \qquad
    \|\widetilde u_j\|_{L^2(O)}\longrightarrow0.
\tag{7.2}
\]
Let \(U_{\mathcal C}^{\times}\) be the punctured cone neighborhood appearing
in \cref{def:admissible-observation}. Since
\(U_{\mathcal C}^{\times}\subset O\),
\[
    \|\widetilde u_j\|_{L^2(U_{\mathcal C}^{\times})}
    \longrightarrow0.
\tag{7.3}\label{eq:dark-measure-pre}
\]
On the smooth manifold
\[
    X^\circ=X\setminus\mathcal C_X,
\]
standard local semiclassical compactness yields, after passage to a
subsequence and diagonal extraction, a positive Radon measure
\(\mu\) on \(S^*X^\circ\). We first verify that no mass is lost at the
deleted cone points. Choose
\(0<\varepsilon'<\varepsilon_{\mathcal C}\) and
\(\chi\in C_c^\infty(X^\circ)\), with \(0\le\chi\le1\), such that
\[
    \chi=1
    \quad\text{when}\quad
    d_X(\,\cdot\,,\mathcal C_X)\ge\varepsilon',
\]
and
\[
    \operatorname{supp}(1-\chi)\setminus\mathcal C_X
    \subset U_{\mathcal C}^{\times}.
\]
Then, \eqref{eq:dark-measure-pre} implies
\[
    \|(1-\chi)\widetilde u_j\|_{L^2(X)}
    \longrightarrow0,
    \qquad
    \|\chi\widetilde u_j\|_{L^2(X)}
    \longrightarrow1.
\]
Passing to the semiclassical limit gives
\[
    \mu(S^*X^\circ)=1.
\]
Thus, no semiclassical mass is lost at the singular set.

We next record the darkness of the limiting measure on the observation set.
For every compact \(K\Subset O\), choose
\(a_K\in C_c^\infty(O)\) with
\(0\le a_K\le1\) and \(a_K=1\) on a neighborhood of \(K\). Then
\[
    0\le
    \langle a_K\widetilde u_j,\widetilde u_j\rangle_{L^2(X)}
    \le
    \|\widetilde u_j\|_{L^2(O)}^2
    \longrightarrow0.
\]
Hence,
\[
    \mu(\pi^{-1}(K))=0
    \qquad\text{for every }K\Subset O.
\]
By inner regularity of the Radon measure \(\mu\),
\[
    \mu(\pi^{-1}(O))=0.
\tag{7.4}\label{eq:dark-measure}
\]
\begin{proposition}[Regular support and invariance]
\label{prop:regular-propagation}
The measure \(\mu\) is supported on
\[
    p^{-1}(0)=S^*X^\circ,
    \qquad
    p(x,\xi)=|\xi|^2-1,
\]
and is invariant under the Hamiltonian flow on every regular flow box
compactly contained in \(X^\circ\).
\end{proposition}

\begin{proof}
Let \(A_h\in\Psi_h^0(X^\circ)\) be compactly supported away from the cone
set. Elliptic regularity applied to
\[
    (-h_j^2\Delta_X-1)\widetilde u_j=0
\]
gives
\[
    \operatorname{supp}\mu\subset\{p=0\}.
\]
If \(B_h\in\Psi_h^0(X^\circ)\) has principal symbol \(b\) and compact support
in a regular flow box, then
\[
    \frac{i}{h_j}
    \left\langle
    [-h_j^2\Delta_X-1,B_{h_j}]
    \widetilde u_j,\widetilde u_j
    \right\rangle
    =0.
\]
The principal symbol of the commutator is \(H_pb\). Passing to the
semiclassical limit therefore gives
\[
    \int_{S^*X^\circ} H_pb\,d\mu=0,
\]
which is the required local invariance.
\end{proof}

\begin{lemma}[Cone-neighborhood removal]
\label{lem:cone-guard-removal}
If a regular phase point reaches \(U_{\mathcal C}^{\times}\) in finite
forward or backward time before meeting a cone point, then it does not belong
to \(\operatorname{supp}\mu\). Consequently, every trajectory contained in
\(\operatorname{supp}\mu\) is a complete regular trajectory that remains a
positive distance from \(\mathcal C_X\).
\end{lemma}

\begin{proof}
Since \(U_{\mathcal C}^{\times}\subset O\),
\eqref{eq:dark-measure} implies that \(\mu\) vanishes on the phase-space lift
of every compact subset of \(U_{\mathcal C}^{\times}\). Let a regular orbit
segment join the given phase point to such a compact subset. Since the
segment is compact and avoids the cone set before reaching the observed
neighborhood, it can be covered by finitely many regular flow boxes.
Successive application of the local invariance in
\cref{prop:regular-propagation} propagates the vanishing of \(\mu\) backward
along this segment. Hence, \(\mu\) vanishes in a neighborhood of the initial
phase point.

Any regular trajectory reaching a cone point must first enter
\(U_{\mathcal C}^{\times}\). Such a trajectory therefore cannot meet
\(\operatorname{supp}\mu\). Since the cone set is finite and
\(\operatorname{supp}\mu\) is closed, every trajectory in the support is
complete and remains a positive distance from \(\mathcal C_X\).
\end{proof}

\begin{lemma}[Directional disintegration]
\label{lem:directional-disintegration}
The translation structure defines a direction map
\[
    \vartheta:S^*X^\circ\longrightarrow\mathbb S^1.
\]
If \(\nu=\vartheta_*\mu\), then
\[
    \mu=\int_{\mathbb S^1}\mu_\theta\,d\nu(\theta),
\]
where, for \(\nu\)-almost every \(\theta\), \(\mu_\theta\) is an invariant
probability measure supported on complete regular trajectories of direction
\(\theta\).
\end{lemma}

\begin{proof}
The measure disintegration theorem applied to the Borel map \(\vartheta\)
gives the stated decomposition. Since direction is constant along every
regular translation trajectory, the Hamiltonian flow preserves the fibers
of \(\vartheta\). Testing the invariance identity of
\cref{prop:regular-propagation} against symbols localized also in the
direction variable and then disintegrating shows that, for
\(\nu\)-almost every \(\theta\), the conditional measure \(\mu_\theta\) is
invariant under the directional flow wherever that flow is regular.
Completeness of the trajectories in its support follows from
\cref{lem:cone-guard-removal}.
\end{proof}

\begin{proposition}[Veech reduction of a dark measure]
\label{prop:directional-reduction}
For \(\nu\)-almost every direction carrying conditional mass, the direction
is completely periodic. In each such direction, the corresponding
conditional measure is supported in the interiors of the maximal periodic
cylinders.
\end{proposition}

\begin{proof}
By \cref{prop:arithmetic-veech}, every direction on \(X\) is either
completely periodic or uniquely ergodic.

Suppose that \(\theta\) is uniquely ergodic and carries conditional mass.
Identifying the fixed-direction fiber of \(S^*X^\circ\) with \(X^\circ\),
regard \(\mu_\theta\) as an invariant Borel probability measure on
\(X^\circ\). By \cref{lem:cone-guard-removal}, its support consists of
complete regular trajectories and does not meet the cone set. Extending
\(\mu_\theta\) by zero across the finite set \(\mathcal C_X\) therefore
gives an invariant Borel probability measure for the directional flow on
\(X\). Unique ergodicity implies that this extension is the normalized flat
area measure. Consequently,
\[
    \mu_\theta(O)
    =
    \frac{\operatorname{Area}(O)}{\operatorname{Area}(X)}
    >0,
\]
since \(O\) is nonempty and open.
On the other hand, disintegrating \eqref{eq:dark-measure} gives
\[
    0
    =
    \mu(\pi^{-1}(O))
    =
    \int_{\mathbb S^1}
    \mu_\theta(\pi^{-1}(O))\,d\nu(\theta).
\]
Because the integrand is nonnegative, it vanishes for
\(\nu\)-almost every \(\theta\). Hence, uniquely ergodic directions carry no
\(\nu\)-mass.

It follows that \(\nu\)-almost every direction carrying conditional mass is
completely periodic. Such a direction decomposes \(X\) into finitely many
maximal periodic cylinders together with their saddle-connection boundaries.
Every saddle connection reaches the cone set at its endpoints and is
therefore excluded from \(\operatorname{supp}\mu\) by
\cref{lem:cone-guard-removal}. Thus, the conditional measure is supported in
the interiors of the maximal periodic cylinders.
\end{proof}

\section{Observability for admissible observation sets}
\label{sec:admissible-observation}

The cone-neighborhood condition has already excluded trajectories reaching
the conic singularities, as well as saddle connections forming the boundaries
of periodic cylinders. It remains to exclude invariant mass supported in the
interiors of such cylinders. The two geometric conditions in the definition
of admissibility and their respective roles are illustrated schematically in
\cref{fig:admissible-geometry}.

\begin{figure}[t]
\centering
\begin{tikzpicture}[scale=1.0,>=Latex,font=\small]
  \draw[line width=1pt,rounded corners] (-5.2,-1.7) rectangle (5.2,1.7);
  \node[above] at (0,1.82) {schematic of the regular part of the branched unfolding $X$};

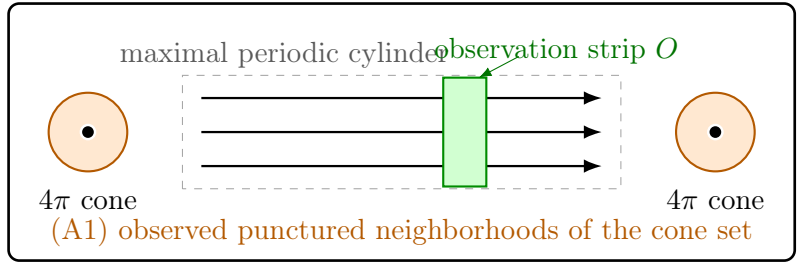
  \fill[black] (-4.15,0) circle (2.2pt);
  \fill[black] (4.15,0) circle (2.2pt);
  \draw[fill=orange!18,draw=orange!70!black,line width=.8pt] (-4.15,0) circle (.52);
  \draw[fill=white,draw=none] (-4.15,0) circle (.11);
  \draw[fill=orange!18,draw=orange!70!black,line width=.8pt] (4.15,0) circle (.52);
  \draw[fill=white,draw=none] (4.15,0) circle (.11);
  \fill[black] (-4.15,0) circle (2.2pt);
  \fill[black] (4.15,0) circle (2.2pt);
  \node[below] at (-4.15,-.62) {$4\pi$ cone};
  \node[below] at (4.15,-.62) {$4\pi$ cone};
  \node[orange!70!black,align=center] at (0,-1.32) {(A1) observed punctured neighborhoods of the cone set};

  \draw[black!35,dashed] (-2.9,-.75) rectangle (2.9,.75);
  \node[black!65] at (-1.55,1.02) {maximal periodic cylinder};
  \foreach \yy in {-0.45,0,0.45} {
    \draw[->,line width=.85pt] (-2.65,\yy) -- (2.65,\yy);
  }
  \draw[fill=green!18,draw=green!55!black,line width=.8pt] (.55,-.72) rectangle (1.12,.72);
  \node[align=left,green!45!black] at (2.05,1.08) {observation strip $O$};
  \draw[->,green!45!black] (1.55,1.00)--(1.02,.72);
  \node[align=center] at (0,-2.25) {(A2) every regular closed orbit in every completely periodic cylinder meets $O$};
\end{tikzpicture}
\caption{Geometric content of the admissible observation class on the branched unfolding.  The two ingredients play different roles: the punctured cone guard removes singular/saddle-boundary trajectories, while periodic-orbit interception excludes invariant mass carried by regular periodic cylinders.  Uniquely ergodic directions require no finite hitting-time condition.  The picture is schematic rather than a polygonal model of $X$.}
\label{fig:admissible-geometry}
\end{figure}


\begin{lemma}[Invariant measures in a periodic cylinder]
\label{lem:admissible-cylinder}
Let
\[
    C\simeq(\mathbb R/L\mathbb Z)_s\times(0,a)_y
\]
be a maximal cylinder in a completely periodic direction, and let
\(O\subset X^\circ\) be open. Suppose that every closed orbit
\[
    (\mathbb R/L\mathbb Z)\times\{y\},
    \qquad y\in(0,a),
\]
meets \(O\). If \(\mu_C\neq0\) is a finite nonnegative measure on \(C\)
invariant under translation in the \(s\)-variable, then
\[
    \mu_C(O\cap C)>0.
\]
\end{lemma}

\begin{proof}
Disintegrate \(\mu_C\) with respect to the transverse coordinate \(y\).
Translation invariance in \(s\) implies that, on every orbit carrying
conditional mass, the conditional measure is the normalized Haar measure on
\(\mathbb R/L\mathbb Z\). Hence
\[
    d\mu_C(s,y)=\frac{ds}{L}\,d\eta(y)
\]
for some finite nonnegative measure \(\eta\) on \((0,a)\).
Define
\[
    \ell_O(y)
    :=
    \mathcal L^1
    \bigl\{s\in\mathbb R/L\mathbb Z:(s,y)\in O\bigr\}.
\]
By the interception hypothesis, every orbit meets \(O\); since \(O\) is open,
this implies
\[
    \ell_O(y)>0
    \qquad\text{for every }y\in(0,a).
\]
Therefore,
\[
    \mu_C(O\cap C)
    =
    \int_{(0,a)}
    \frac{\ell_O(y)}{L}\,d\eta(y).
\]
If \(\mu_C\neq0\), then \(\eta\neq0\), and the strictly positive integrand
gives
\[
    \mu_C(O\cap C)>0.
\]
\end{proof}

\begin{theorem}[High-frequency observability for admissible sets]
\label{thm:high-frequency-admissible}
For every \(W\in\mathcal O_{\rm adm}(\T)\), there exist constants
\(h_0>0\) and \(c_{\rm HF}(W)>0\) such that every exact DDN or NND
eigenfunction \(u_h\) with \(0<h<h_0\) satisfies
\[
    \|u_h\|_{L^2(W)}^2
    \ge
    c_{\rm HF}(W)\|u_h\|_{L^2(\T)}^2.
\]
\end{theorem}

\begin{proof}
Suppose, to the contrary, that the estimate fails, and consider the normalized
unfolded contradiction sequence constructed in
\cref{sec:high-frequency}. Its semiclassical measure \(\mu\) has total mass
one and satisfies
\[
    \mu(\pi^{-1}(O))=0,
    \qquad O=\widetilde W.
\]
By \cref{prop:directional-reduction}, for \(\nu\)-almost every direction
carrying conditional mass, the direction is completely periodic and the
conditional probability measure \(\mu_\theta\) is supported in the interiors
of finitely many maximal periodic cylinders. Since \(\mu_\theta\) has total
mass one, at least one of its cylinder restrictions is nonzero.

Condition \textnormal{(A2)} in
\cref{def:admissible-observation} states that every regular closed orbit in
each such cylinder meets \(O\). After identifying the fixed-direction phase
cylinder with its base cylinder, \cref{lem:admissible-cylinder} therefore
implies
\[
    \mu_\theta(\pi^{-1}(O))>0
\]
for \(\nu\)-almost every direction carrying conditional mass. Since
\(\nu=\vartheta_*\mu\) is a probability measure, directional disintegration
then yields
\[
    \mu(\pi^{-1}(O))
    =
    \int_{\mathbb S^1}
    \mu_\theta(\pi^{-1}(O))\,d\nu(\theta)
    >0.
\]
This contradicts \eqref{eq:dark-measure}. Hence, no normalized dark
high-frequency eigensequence exists, and the claimed estimate follows.
\end{proof}

\begin{proposition}[Structure of the admissible class]
\label{prop:admissible-class}
The class \(\mathcal O_{\rm adm}(X)\) is upward closed among open subsets of
\(X^\circ\). Moreover, let
\[
    K=K_1\cup\cdots\cup K_m\Subset X^\circ
\]
be a finite union of pairwise disjoint compact sets. Suppose that each
\(K_j\) is contained in an embedded Euclidean disk
\[
    K_j\subset B_j\Subset X^\circ,
\]
where the disks \(B_j\) are pairwise disjoint. Then,
\[
    O:=X^\circ\setminus K
\]
belongs to \(\mathcal O_{\rm adm}(X)\).
\end{proposition}

\begin{proof}
Upward closure follows immediately from conditions
\textnormal{(A1)} and \textnormal{(A2)}.
Since \(K\Subset X^\circ\) and the cone set \(\mathcal C_X\) is finite,
\[
    d_X(K,\mathcal C_X)>0.
\]
Hence, \(O=X^\circ\setminus K\) contains a fixed punctured neighborhood of
\(\mathcal C_X\), so condition \textnormal{(A1)} holds.
Suppose, for contradiction, that a regular closed geodesic \(\gamma\) is
disjoint from \(O\). Then,
\[
    \gamma(\mathbb S^1)\subset K_1\cup\cdots\cup K_m.
\]
Since \(\gamma(\mathbb S^1)\) is connected and the finitely many compact sets
\(K_j\) are pairwise disjoint, its image must be contained in a single
\(K_j\). Consequently,
\[
    \gamma(\mathbb S^1)\subset B_j
\]
for some embedded Euclidean disk \(B_j\). Developing \(B_j\) isometrically
into \(\mathbb R^2\) sends \(\gamma\) to a nonconstant closed straight
geodesic segment in the Euclidean plane, which is impossible. Therefore,
every regular closed orbit meets \(O\), and condition
\textnormal{(A2)} follows.
\end{proof}

\begin{corollary}[A concrete admissible class on the triangle]
\label{cor:triangle-admissible-subclass}
Let \(K_T\Subset\T\) be compact. Suppose that the finitely many distinct
lifts of \(K_T\) to \(X^\circ\) are pairwise disjoint and that each is
contained in one of a family of pairwise disjoint embedded Euclidean disks.
Then,
\[
    W:=\T\setminus K_T
\]
belongs to \(\mathcal O_{\rm adm}(\T)\) and therefore satisfies the uniform
DDN/NND observability estimate of \cref{thm:mixed-uniform}.
\end{corollary}

\begin{proof}
By definition of the full lift,
\[
    X^\circ\setminus\widetilde W
    =
    F^{-1}(K_T)\cap X^\circ,
\]
where \(F:X\to\T\) is the folding map. This set is precisely the finite union
of the distinct lifted copies of \(K_T\). The hypotheses therefore place
\(X^\circ\setminus\widetilde W\) in the class covered by
\cref{prop:admissible-class}. Hence,
\(\widetilde W\in\mathcal O_{\rm adm}(X)\), and thus
\(W\in\mathcal O_{\rm adm}(\T)\).
The conclusion follows from \cref{thm:mixed-uniform}.
\end{proof}

\begin{proof}[Proof of \cref{cor:small-hole-admissible}]
The unfolding involves only finitely many triangle-copy embeddings. After
discarding redundant copies, let
\(\iota_\alpha,\iota_\beta\) be two distinct such embeddings. Their
coincidence set
\[
    Z_{\alpha\beta}
    :=
    \{x\in\T^\circ:
      \iota_\alpha(x)=\iota_\beta(x)\}
\]
has empty interior: if the two local Euclidean isometries agreed on a
nonempty open subset, they would agree identically on their common domain,
contrary to their being distinct.

Each \(Z_{\alpha\beta}\) is closed in \(\T^\circ\). Hence, the finite union
\[
    Z:=\bigcup_{\alpha\neq\beta}Z_{\alpha\beta}
\]
is closed with empty interior, and \(\T^\circ\setminus Z\) is dense and open.
For every \(x_0\in\T^\circ\setminus Z\), the finitely many lifts of \(x_0\)
are distinct, have positive mutual separation, and lie a positive distance
from \(\mathcal C_X\). It follows that, for all sufficiently small \(r>0\),
the distinct lifts of
\[
    \overline{B_\T(x_0,r)}
\]
are contained in pairwise disjoint embedded Euclidean disks. Applying
\cref{cor:triangle-admissible-subclass} gives the small-hole conclusion, and
\cref{thm:mixed-uniform} gives the corresponding observability estimate.

The finite-hole statement follows by the same argument applied simultaneously
to the finite collection of all lifted holes, provided \(r\) is chosen small
enough that the resulting lifted compact sets remain mutually separated and
away from the cone set.
\end{proof}

\begin{remark}[Breadth of the admissible class]
\label{rem:admissible-breadth}
The preceding proposition gives a concrete nonempty family of admissible
observation sets: the lifted unobserved region may be any finite union of
separated compact subsets contained in regular Euclidean disks. Since
\(\mathcal O_{\rm adm}(X)\) is upward closed, every enlargement of such an
observation set remains admissible. More generally, admissibility allows
observation geometries that contain a punctured neighborhood of the cone set
and intersect every regular closed orbit in each completely periodic
direction; no corresponding hitting condition is imposed in uniquely
ergodic directions.
\end{remark}

\section{From high frequency to the full spectrum}
\label{sec:globalization}

\begin{proposition}[Low-frequency closure]
\label{prop:globalization}
Fix \(B\in\{DDN,NND\}\) and a nonempty open set \(W\subset\T\).
Suppose that there exist \(\lambda_0>0\) and \(c_{\rm HF}>0\) such that
every \(B\)-eigenfunction \(u\) with eigenvalue \(\lambda\ge\lambda_0\)
satisfies
\[
    \|u\|_{L^2(W)}^2
    \ge
    c_{\rm HF}\|u\|_{L^2(\T)}^2.
\]
Then, there exists \(c_W>0\) such that every \(B\)-eigenfunction \(u\)
satisfies
\[
    \|u\|_{L^2(W)}^2
    \ge
    c_W\|u\|_{L^2(\T)}^2.
\]
\end{proposition}

\begin{proof}
The mixed Laplacian on the bounded triangle \(\T\) has compact resolvent, so
only finitely many eigenvalues lie below \(\lambda_0\). Fix one such
eigenvalue \(\lambda\). The unit sphere of the finite-dimensional eigenspace
\(E_\lambda^B\) is compact, and the map
\[
    u\longmapsto \|u\|_{L^2(W)}^2
\]
is continuous. Hence, it attains its minimum on that sphere.
This minimum is strictly positive. Indeed, if it were zero, there would exist
\(0\neq u\in E_\lambda^B\) with
\[
    \|u\|_{L^2(W)}=0.
\]
Thus, \(u=0\) almost everywhere on \(W\). Since \(u\) solves
\[
    (-\Delta-\lambda)u=0
\]
in the interior of \(\T\), interior elliptic regularity gives a smooth
representative there, and hence \(u\) vanishes on the nonempty open set
\(W\). Interior unique continuation then implies
\[
    u\equiv0
    \qquad\text{on }\T,
\]
contradicting \(\|u\|_{L^2(\T)}=1\).
Therefore, every eigenspace below \(\lambda_0\) has a strictly positive
observation minimum. Taking the minimum of these finitely many constants and
\(c_{\rm HF}\) yields a constant \(c_W>0\) valid over the entire
\(B\)-spectrum.
\end{proof}

\begin{proof}[Proof of \cref{thm:mixed-uniform}]
Let \(W\in\mathcal O_{\rm adm}(\T)\) and fix
\(B\in\{DDN,NND\}\). By
\cref{thm:high-frequency-admissible}, the required observation estimate holds
uniformly for all sufficiently high-frequency \(B\)-eigenfunctions.
Applying \cref{prop:globalization} extends the estimate to the full
\(B\)-spectrum. Since the argument applies to both \(B=DDN\) and \(B=NND\),
the theorem follows.
\end{proof}

\section{Recombining the parity sectors}
\label{sec:parity-recombination}

\begin{proof}[Proof of \cref{thm:reflection-overlap}]
Let
\[
    V=\operatorname{int}(W\cap\sigma(W)).
\]
By hypothesis, \(V\) is nonempty, open, and invariant under the reflection
\(\sigma\). Decompose the rhombus eigenfunction orthogonally into its even and
odd parts,
\[
    u=u_++u_-,
    \qquad
    u_\pm\circ\sigma=\pm u_\pm.
\]
Upon restriction to one half-triangle \(\T\), these components correspond to
the DDD and DDN sectors in the Dirichlet case, and to the NNN and NND sectors
in the Neumann case.

For the pure DDD/NNN sector, the equilateral-triangle observability theorem
of Alphonse--Lafontaine gives a frequency-uniform stationary observation
estimate on every nonempty open subset
\cite{AlphonseLafontaine2025}. For the mixed sector, the assumption
\[
    V\cap\T\in\mathcal O_{\rm adm}(\T)
\]
allows us to apply \cref{thm:mixed-uniform}. Hence, there exist constants
\(c_+,c_->0\), independent of the eigenvalue, such that
\[
    \|u_+\|_{L^2(V\cap\T)}^2
    \ge c_+\|u_+\|_{L^2(\T)}^2,
    \qquad
    \|u_-\|_{L^2(V\cap\T)}^2
    \ge c_-\|u_-\|_{L^2(\T)}^2,
\]
with the assignment of the pure and mixed constants determined by the
Dirichlet or Neumann case. Since \(V\) is reflection invariant and \(u_+\) and \(u_-\) have opposite
parity,
\[
    \int_V u_+\overline{u_-}
    =
    \int_V
    (u_+\circ\sigma)\,
    \overline{(u_-\circ\sigma)}
    =
    -\int_V u_+\overline{u_-}.
\]
Therefore,
\[
    \int_V u_+\overline{u_-}=0
\]
and
\[
    \|u\|_{L^2(V)}^2
    =
    \|u_+\|_{L^2(V)}^2
    +
    \|u_-\|_{L^2(V)}^2.
\]
Reflection symmetry also gives
\[
    \|u_\pm\|_{L^2(V)}^2
    =
    2\|u_\pm\|_{L^2(V\cap\T)}^2,
    \qquad
    \|u_\pm\|_{L^2(\Rh)}^2
    =
    2\|u_\pm\|_{L^2(\T)}^2.
\]
Consequently, with \(c=\min\{c_+,c_-\}>0\),
\[
\begin{aligned}
    \|u\|_{L^2(V)}^2
    &=
    \|u_+\|_{L^2(V)}^2
    +
    \|u_-\|_{L^2(V)}^2
    \\
    &\ge
    c\bigl(
        \|u_+\|_{L^2(\Rh)}^2
        +
        \|u_-\|_{L^2(\Rh)}^2
      \bigr).
\end{aligned}
\]
The parity decomposition is orthogonal on \(\Rh\), so
\[
    \|u_+\|_{L^2(\Rh)}^2
    +
    \|u_-\|_{L^2(\Rh)}^2
    =
    \|u\|_{L^2(\Rh)}^2.
\]
Thus,
\[
    \|u\|_{L^2(V)}^2
    \ge
    c\|u\|_{L^2(\Rh)}^2.
\]
Since \(V\subset W\),
\[
    \|u\|_{L^2(W)}^2
    \ge
    \|u\|_{L^2(V)}^2,
\]
and the theorem follows.
\end{proof}

\begin{proof}[Proof of \cref{prop:fixed-angle}]
Fix
\[
    (B_+,B_-)\in\{(DDN,DDD),(NNN,NND)\}
\]
Let \(\lambda\) be an eigenvalue common to the two sectors. The corresponding
eigenspaces are finite dimensional. By interior unique continuation,
restriction to the nonempty open set \(W\) is injective on each eigenspace.
Thus,
\[
    u\longmapsto\|u\|_{L^2(W)}
\]
defines a norm on each of them. Since all norms on a finite-dimensional
vector space are equivalent, the corresponding unit spheres are compact.
It follows that the quotient defining
\(\rho_W^{B_+,B_-}(\lambda)\) attains its supremum.

Suppose, for contradiction, that
\[
    \rho_W^{B_+,B_-}(\lambda)=1.
\]
Equality in the Cauchy--Schwarz inequality then yields nonzero eigenfunctions
\(u_+\) and \(u_-\), together with a scalar \(c\neq0\), such that
\[
    u_+=c\,u_-
    \qquad\text{almost everywhere on }W.
\]
The difference
\[
    v=u_+-c\,u_-
\]
satisfies
\[
    (-\Delta-\lambda)v=0
\]
in the interior of \(\T\) and vanishes on the nonempty open set \(W\).
Interior unique continuation therefore gives
\[
    v\equiv0
    \qquad\text{on }\T.
\]
Hence,
\[
    u_+=c\,u_-
    \qquad\text{throughout }\T.
\]
The two sectors \(B_+\) and \(B_-\) impose opposite boundary conditions on
the symmetry side \(\M\): one imposes Dirichlet data and the other Neumann
data. Thus the common nonzero solution obtained above satisfies, on the
interior of \(\M\),
\[
    u|_{\M}=0,
    \qquad
    \partial_\nu u|_{\M}=0.
\]
The Cauchy uniqueness theorem for the constant-coefficient Helmholtz equation
across the straight segment \(\M\) then implies
\[
    u\equiv0
    \qquad\text{on }\T,
\]
contradicting the choice of \(u_+\) and \(u_-\). Therefore,
\[
    \rho_W^{B_+,B_-}(\lambda)<1.
\]
\end{proof}

\section{The arbitrary-open problem and the saddle-network obstruction}
\label{sec:endpoint}

The unconditional results above identify and eliminate the regular-flow
localization mechanisms for admissible observation sets. For a completely
arbitrary nonempty open set, the same analysis isolates one additional
mechanism rather than producing a second unconditional observability theorem.
Namely, concentration near the conic points is no longer excluded by the
geometric hypothesis used in \cref{sec:high-frequency}.

The regular component of a high-frequency sequence with vanishing observed
mass can nevertheless be reduced to completely periodic cylinders, and a
fixed-collar estimate transfers the remaining cylinder mass toward their
saddle boundaries. Thus any residual obstruction is localized to the
saddle/conic network. Extending the main theorem to arbitrary open sets is
therefore reduced to a uniform exact-spectral transfer estimate from this
network to an observed regular region, or to a regular component already
controlled by the observation estimate. The precise reduction is given in
\cref{app:endpoint-reduction}. The criterion below records a sufficient
quantitative closure condition; it is not assumed anywhere in the
unconditional results of the paper.

\begin{definition}[Quantitative saddle-network exclusion]
\label{def:qsn-exclusion}
We say that the mixed-sector problem satisfies \emph{quantitative
saddle-network exclusion} if there exist constants
\[
    \delta>0
    \qquad\text{and}\qquad
    N\in\mathbb N
\]
with the following property. Consider any normalized exact mixed-sector
high-frequency sequence whose observed mass tends to zero, and any nonzero
residual microlocal component that, after the periodic-cylinder reduction,
reaches a saddle/conic network. Within at most \(N\) successive geometric or
conic returns, at least a fraction \(\delta\) of the normalized microlocal mass
or flux of that component is transferred either to an observed regular tube
or to a regular component already controlled by the observation estimate.
The constants \(\delta\) and \(N\) are uniform along the exact eigensequence.
\end{definition}

\begin{remark}[Role of the saddle-network hypothesis]
\label{rem:qsn-nature}
The preceding definition supplies the quantitative input needed after the
regular directional and periodic-cylinder reductions have been exhausted.
Those reductions localize any remaining obstruction to the saddle/conic
network, while quantitative saddle-network exclusion prevents a nonzero
residual component from remaining indefinitely confined to that network.
Thus the hypothesis concerns only the residual mechanism that arises when the
observation set does not contain a fixed punctured neighborhood of the cone
set.
\end{remark}

\begin{proposition}[A sufficient closure criterion for arbitrary-open observability]
\label{prop:mixed-uniform-arbitrary}
Assume quantitative saddle-network exclusion. Then, for every nonempty open
set \(W\subset\T\) and every \(B\in\{DDN,NND\}\), there exists a constant
\(c_W>0\) such that
\[
    \|u\|_{L^2(W)}^2
    \ge
    c_W\|u\|_{L^2(\T)}^2
\]
for every \(B\)-eigenfunction \(u\), with \(c_W\) independent of the
eigenvalue.
\end{proposition}

\begin{proof}
Suppose that the high-frequency estimate fails. Then, there exists a normalized
exact mixed-sector eigensequence whose mass on \(W\) tends to zero. By the
arbitrary-open directional reduction in
\cref{prop:directional-reduction-arbitrary} and the fixed-collar
periodic-cylinder estimate in
\cref{prop:periodic-cylinder-control}, every nontrivial component that survives
the regular-flow analysis is transferred to the associated saddle/conic
network.

Quantitative saddle-network exclusion provides constants \(\delta>0\) and
\(N<\infty\), uniform along the eigensequence, such that within at most \(N\)
returns a fixed positive fraction of each nonzero residual component reaches
either an observed regular tube or a regular component already controlled by
the observation estimate. In the first case, the observed mass is bounded
away from zero; in the second, the established regular-region estimate yields
the same conclusion. Both alternatives contradict the assumed vanishing of
the observed mass. Hence, no such high-frequency eigensequence exists.

The high-frequency estimate extends to the full spectrum by
\cref{prop:globalization}.
\end{proof}

\begin{remark}[Conditional full-rhombus consequence under finite spectral coincidence]
\label{rem:finite-coincidence}
Assume quantitative saddle-network exclusion. Consider either the Dirichlet
rhombus problem, with triangle parity pair \((DDN,DDD)\), or the Neumann
rhombus problem, with triangle parity pair \((NNN,NND)\). If the two sectors
in the corresponding pair have only finitely many common eigenvalues, then
the full-rhombus Laplacian satisfies frequency-uniform observability on every
nonempty open set \(W\subset\Rh\).
\end{remark}

\begin{proof}[Justification]
Let \(W\subset\Rh\) be nonempty and open. Choose a half-triangle \(\T\) such
that
\[
    W_T:=W\cap\T
\]
is nonempty. For the mixed member of the corresponding parity pair,
\cref{prop:mixed-uniform-arbitrary} gives a frequency-uniform observation
estimate on \(W_T\). For the pure member, the equilateral-triangle
observability theorem of Alphonse--Lafontaine provides the corresponding
estimate \cite{AlphonseLafontaine2025}.

At an eigenvalue belonging to only one parity sector, the sectorwise estimate,
together with the unitary restriction map of
\cref{prop:parity-decomposition}, gives the required full-rhombus bound.
Suppose now that \(\lambda\) is a common eigenvalue and write
\(u=u_++u_-\) according to the corresponding parity decomposition. By
\cref{prop:fixed-angle},
\[
    \rho_{W_T}^{B_+,B_-}(\lambda)<1,
\]
and therefore
\[
\begin{aligned}
    \|u_++u_-\|_{L^2(W_T)}^2
    &\ge
    \bigl(1-\rho_{W_T}^{B_+,B_-}(\lambda)\bigr)
    \left(
        \|u_+\|_{L^2(W_T)}^2
        +
        \|u_-\|_{L^2(W_T)}^2
    \right).
\end{aligned}
\]
Applying the two sectorwise observation estimates and then the unitary
restriction maps controls the corresponding global parity norms.

Let \(\Lambda_{\rm com}\) denote the set of common eigenvalues. If
\(\Lambda_{\rm com}\neq\varnothing\), the finiteness assumption gives
\[
    \delta_W
    :=
    \min_{\lambda\in\Lambda_{\rm com}}
    \left(
        1-\rho_{W_T}^{B_+,B_-}(\lambda)
    \right)
    >0.
\]
Combining this angle gap with the two sectorwise observation constants yields
\[
    \|u\|_{L^2(W_T)}^2
    \ge
    c_W\|u\|_{L^2(\Rh)}^2
\]
with \(c_W>0\) independent of the eigenvalue. If
\(\Lambda_{\rm com}=\varnothing\), the same conclusion follows directly from
the sectorwise estimates. Since \(W_T\subset W\), the estimate also holds on
\(W\).
\end{proof}
\begin{remark}[Open endpoint]
\label{rem:open-endpoint}
The quantitative saddle-network exclusion hypothesis is not established in
the present work. Removing this hypothesis requires uniform exact-spectral
control of residual mass propagating through the saddle/conic network,
including along sequences for which the relevant directions and microlocal
networks vary with the eigenvalue. Such a result would remove the remaining
restriction in the arbitrary-open argument and extend the mixed-sector
observability theorem from the admissible class to every nonempty open
observation set.
\end{remark}


\section{Discussion}
\label{sec:discussion}

The main phenomenon isolated in this work is that mixed reflection monodromy
can change the mechanism of eigenfunction observability. When the boundary
signs form a scalar reflection character, unfolding may place the problem in
a toric setting. For the DDN/NND sectors considered here, that closure fails.
The correct scalar phase space is instead a branched genus-two translation
surface, and frequency-uniform non-localization is governed by the
directional dynamics of this higher-genus surface.

The proof reveals two distinct mechanisms by which high-frequency mass could
avoid an observation region. One is concentration approaching the
saddle/conic network; the other is invariant concentration inside periodic
cylinders. The admissibility condition has precisely these two components.
The punctured-cone guard controls the first mechanism, while periodic-orbit
interception controls the second. In uniquely ergodic directions no
additional geometric interception is required: invariance and the Veech
dichotomy already preclude a dark component.

This separation is useful beyond the particular proof because it identifies
which part of the argument is dynamical and which part remains genuinely
exact-spectral. The regular-flow problem is resolved by semiclassical
propagation and arithmetic directional dynamics. For arbitrary open
observations, the unresolved endpoint is concentrated entirely at the
saddle/conic network, where one needs a quantitative transfer estimate
uniform over exact eigensequences whose microlocal scales and relevant
directions may vary with frequency. Section~\ref{sec:endpoint} formulates a
sufficient quantitative closure criterion, but that criterion is not
established in the present work and is not part of the unconditional theorem.

The resulting admissible class is nevertheless nonempty, stable under
enlargement, and contains explicit families of observation sets. In
particular, \cref{cor:small-hole-admissible} shows that complements of
sufficiently small regular interior holes provide concrete examples. The
same mixed-sector estimate also contributes to observability on the original
rhombus. For reflection-overlapping observation sets, parity recombination
combines the mixed-sector theorem with observability in the pure
equilateral-triangle sectors, while reflection symmetry removes cross-parity
cancellation on the overlap region.

Two natural extensions emerge. The first is to establish exact-spectral
control of the saddle/conic network and thereby remove the punctured-cone
condition from the mixed-sector theorem. The second is to control
cross-parity cancellation uniformly in frequency for general observation
sets on the rhombus, beyond the reflection-overlap geometry treated here.
These are logically distinct from the present theorem and identify the
remaining steps toward frequency-uniform observability from arbitrary
nonempty open subsets of the full equilateral rhombus.


\appendix

\section{Mixed-corner local model}
\label{app:conic}
Consider first a D/N corner of opening \(\pi/3\).  Separation of variables
\(u(r,\varphi)=R(r)\Phi(\varphi)\) gives
\[
 -\Phi''=\nu^2\Phi,
 \qquad \Phi(0)=0,
 \qquad \Phi'(\pi/3)=0.
\]
Thus, \(\Phi(\varphi)=\sin(\nu\varphi)\) and
\(\cos(\nu\pi/3)=0\), so
\[
 \nu_k=3k+\frac32,\qquad k=0,1,2,\ldots.
\tag{A.1}\label{eq:mixed-angular-spectrum}
\]

\begin{lemma}[DDN/NND local equivalence]
\label{lem:dn-nd-equivalence}
The D/N and N/D wedge Laplacians are unitarily equivalent.  In particular they
have the same angular spectrum \eqref{eq:mixed-angular-spectrum}, the same Friedrichs radial orders,
and the same radial Green-current normalization.
\end{lemma}

\begin{proof}
Let \(J:L^2(0,\pi/3)\to L^2(0,\pi/3)\) be reflection of the angular variable,
\((Jf)(\varphi)=f(\pi/3-\varphi)\).  If \(f(0)=0\) and
\(f'(\pi/3)=0\), then \((Jf)'(0)=0\) and \((Jf)(\pi/3)=0\).  Hence, \(J\)
interchanges the D/N and N/D domains and commutes with \(-\partial_\varphi^2\).
Tensoring with the identity in the radial variable gives the unitary
equivalence of the wedge Laplacians.  Since \(J\) acts only in the angular variable and is unitary, it also
preserves the radial Green current.
\end{proof}

\subsection{The \texorpdfstring{$4\pi$}{4pi} cone and the mixed character}
After the two-sheeted monodromy resolution, a mixed vertex becomes a flat cone
with link \(Y=\mathbb R/(4\pi\mathbb Z)\).  The D/N reflection character picks
out the angular modes that are anti-periodic under rotation by \(2\pi/3\); in
Fourier notation these are the modes \(n\equiv3\pmod 6\), whose absolute link
frequencies are \(|n|/2=3k+3/2\).

For the Friedrichs Laplacian on the exact product cone, the stationary
scattering operator on the link is, up to a global phase,
\[
 S_c=-i\exp\!\bigl(-i\pi\sqrt{\Delta_Y}\bigr),
\tag{A.2}\label{eq:cone-scattering}
\]
see \cite{Yang2022ProductCones}.  On the mixed-character modes its multiplier
has modulus one; with the convention \eqref{eq:cone-scattering}, the multiplier
on the positive mode \(n=6k+3\) is \((-1)^k\).

\begin{proposition}[Local cone unitarity does not imply aperture coercivity]
\label{prop:local-cone-no-go}
Let \(\mathcal H_{\rm mix}\subset L^2(Y)\) be the mixed-character subspace and
let \(\Gamma\subset Y\) be a proper open aperture.  Assume that there exists a
nonempty open interval \(I\subset Y\) whose full orbit under the finite
rotation symmetry defining \(\mathcal H_{\rm mix}\) is disjoint from
\(\Gamma\).  Then, there is no \(c>0\) such that
\[
 \|1_\Gamma S_ca\|\ge c\|a\|,
 \qquad a\in\mathcal H_{\rm mix}.
\]
\end{proposition}

\begin{proof}
Choose a nonempty interval \(I_0\Subset I\) so small that its distinct
translates under the finite rotation symmetry are pairwise disjoint, and take
\(g_0\in C_c^\infty(I_0)\), \(g_0\ne0\).  Extend \(g_0\) to the finite
rotation orbit of \(I_0\) using the mixed sign character.  The resulting
\(g\in\mathcal H_{\rm mix}\) is nonzero and vanishes on \(\Gamma\).  Since
\(S_c\) is unitary and preserves \(\mathcal H_{\rm mix}\), set
\(a=S_c^{-1}g\).  Then \(a\ne0\) while
\(1_\Gamma S_ca=1_\Gamma g=0\).
\end{proof}

\begin{remark}[Geometric and strictly diffractive propagation]
Semiclassical propagation through cone points distinguishes the geometric
relation from strictly diffractive propagation, the latter exhibiting
improved regularity; see Melrose--Wunsch and Hintz
\cite{MelroseWunsch2004,Hintz2024Cone}. This distinction is relevant to the
arbitrary-open extension, where concentration may interact with the conic set
without being excluded geometrically by the observation region. For the
admissible class considered in \cref{thm:mixed-uniform}, the punctured cone
guard removes this residual conic mechanism before the regular-flow analysis
is applied.
\end{remark}


\section{Exact-eigenfunction compatibility}
\label{app:realizability}

Microlocal channel data extracted from a single exact eigenfunction inherit
global compatibility relations from the underlying scalar solution. Thus, for
any fixed finite microlocal decomposition, the local channel components cannot
be prescribed independently: propagation, reflection, branch identification,
and monodromy relate the traces associated with different sections. We record
this qualitative realizability property below.

\begin{lemma}[Compatibility of realizable channel data]
\label{lem:global-compatibility}
Let
\[
    A_h(u_h)=(a_{\alpha,h})_{\alpha\in\mathcal A}
\]
be a fixed finite family of microlocal traces extracted from a single exact
unfolded eigenfunction \(u_h\). Whenever two channel sections are connected by
a propagation step in the chosen finite microlocal construction, their traces
satisfy the corresponding propagation relation, modulo the remainders
introduced by the microlocal cutoffs and parametrices.

If two admissible propagation histories terminate at the same microlocal
section, the Cauchy data obtained along the two histories agree modulo the
accumulated remainders. Consequently, \(A_h(u_h)\) belongs to the
compatibility class determined by the regular propagation, reflection,
branch-identification, monodromy, and, whenever present in the construction,
conic-transition relations.
\end{lemma}

\begin{proof}
Every component \(a_{\alpha,h}\) is obtained by applying a fixed microlocal
trace operator to the same exact eigenfunction \(u_h\). Along a regular segment
joining two channel sections, propagation of solutions to the eigenvalue
equation identifies the corresponding microlocal Cauchy data, modulo the
remainders introduced by localization and the chosen parametrix. The same
statement applies to a conic transition whenever such a transition is included
in the finite construction.

Now, consider two admissible propagation histories terminating at the same
microlocal section. Both histories reconstruct microlocal Cauchy data of the
same global scalar solution \(u_h\) at that section. Their difference is
therefore accounted for by the remainders accumulated along the two
constructions. The reflection and branch-identification relations follow from
the exact mixed-sector unfolding, while continuation around closed reflection
histories gives the corresponding monodromy relations. These identities show
that the family \(A_h(u_h)\) lies in the stated compatibility class.
\end{proof}

\begin{remark}[Compatibility and quantitative control]
\label{rem:compatibility-not-coercivity}
The lemma identifies the constrained class of channel data realizable by a
single exact eigenfunction. Passing from this qualitative compatibility to
arbitrary-open observability requires quantitative control that is uniform in
the high-frequency limit. In particular, pointwise injectivity of a
finite-dimensional compatibility map need not yield a uniform lower bound as
\(h\to0\), since its smallest singular value may degenerate. The quantitative
saddle-network exclusion property in \cref{def:qsn-exclusion} formulates the
additional uniform control required at this stage.
\end{remark}


\section{Reduction for the arbitrary-open endpoint}
\label{app:endpoint-reduction}

This appendix records the regular directional and periodic-cylinder reductions
used in the arbitrary-open extension of
\cref{prop:mixed-uniform-arbitrary}. They identify the part of a possible
high-frequency concentration mechanism that remains after the regular-flow
analysis.

\begin{proposition}[Regular directional reduction for arbitrary open sets]
\label{prop:directional-reduction-arbitrary}
Let \(W\subset\T\) be nonempty and open, and let \(u_j\) be a normalized
exact mixed-sector eigensequence such that
\[
    \|u_j\|_{L^2(W)}\longrightarrow0.
\]
Let \(\mu\) be an associated semiclassical measure on the regular phase space
\(S^*X^\circ\). Then every invariant component of \(\mu\) supported on
complete regular trajectories in a uniquely ergodic direction has zero mass.
Consequently, every nonzero component supported on complete regular
trajectories is carried by the interiors of maximal cylinders in completely
periodic directions. Components associated with trajectories approaching the
cone set are not eliminated by this regular reduction and form the residual
conic component.
\end{proposition}

\begin{proof}
The support and local invariance argument of
\cref{prop:regular-propagation} applies in every regular flow box compactly
contained in \(X^\circ\). Restrict \(\mu\) to its invariant part supported on
complete regular trajectories and disintegrate this part with respect to the
translation direction.

Suppose that a uniquely ergodic direction carries nonzero conditional mass.
After normalization, the corresponding conditional measure is an invariant
probability measure supported on complete regular trajectories. Since the cone
set is finite and the component under consideration is supported on complete
regular trajectories, extending this measure by zero across the cone points
gives an invariant probability measure for the directional flow on \(X\).
Unique ergodicity therefore identifies it with normalized flat area measure.
On the other hand, the full lift
\[
    O=\widetilde W
\]
is nonempty and open, and hence
\[
    \frac{\operatorname{Area}(O)}{\operatorname{Area}(X)}>0.
\]
This contradicts the vanishing of the limiting observed mass on \(O\).

By the Veech dichotomy of \cref{prop:arithmetic-veech}, every remaining
direction carrying complete regular mass is therefore completely periodic.
In such a direction, the regular part of \(X\) decomposes into maximal
periodic cylinders separated by saddle connections. A complete regular
trajectory that avoids the cone set lies in the interior of one of these
cylinders. The components not covered by this argument are precisely those
whose propagation interacts with the conic set, and these constitute the
residual conic component.
\end{proof}

\begin{lemma}[Uniform two-collar estimate]
\label{lem:uniform-two-collar}
Let \(I_0\Subset(0,a)\), and let
\[
    I_\omega=(0,\varepsilon)\cup(a-\varepsilon,a),
    \qquad 0<\varepsilon<\frac a2.
\]
There exists a constant \(C=C(a,I_0,\varepsilon)>0\) such that, for every
\(h>0\), every \(\mu\in\mathbb R\), and every complex-valued solution \(v\)
of
\[
    -h^2v''+\mu v=0
    \qquad\text{on }(0,a),
\]
one has
\[
    \|v\|_{L^2(I_0)}
    \le
    C\|v\|_{L^2(I_\omega)}.
\]
The constant is independent of \(h\), \(\mu\), and \(v\).
\end{lemma}

\begin{proof}
Set \(\alpha=\mu/h^2\). We treat the oscillatory, hyperbolic, and bounded
transition regimes.

Suppose first that \(\alpha=-k^2<0\) and \(k\ge 2/\varepsilon\). Then
\[
    v(x)=Ae^{ikx}+Be^{-ikx}.
\]
Since \(I_\omega\) is the union of two intervals of length \(\varepsilon\),
\[
    \left|\int_{I_\omega}e^{2ikx}\,dx\right|\le \frac{2}{k}.
\]
Consequently,
\[
\begin{aligned}
    \|v\|_{L^2(I_\omega)}^2
    &=2\varepsilon\bigl(|A|^2+|B|^2\bigr)
      +2\operatorname{Re}\!\left(
          A\overline B\int_{I_\omega}e^{2ikx}\,dx
        \right)\\
    &\ge
      \left(2\varepsilon-\frac{2}{k}\right)
      \bigl(|A|^2+|B|^2\bigr)
     \ge
      \varepsilon\bigl(|A|^2+|B|^2\bigr).
\end{aligned}
\]
On the other hand,
\[
    \|v\|_{L^2(I_0)}^2
    \le 2|I_0|\bigl(|A|^2+|B|^2\bigr),
\]
which gives the required estimate uniformly in this regime.
Next suppose that \(\alpha=k^2>0\). Write
\[
    v(x)=p\,e^{-k(a-x)}+q\,e^{-kx},
\]
and set
\[
    d_0:=\operatorname{dist}(I_0,\{0,a\})>0.
\]
For \(x\in I_0\),
\[
    |v(x)|^2
    \le 2e^{-2kd_0}\bigl(|p|^2+|q|^2\bigr),
\]
so
\[
    \|v\|_{L^2(I_0)}^2
    \le
    2|I_0|e^{-2kd_0}\bigl(|p|^2+|q|^2\bigr).
\tag{A.3}\label{eq:two-collar-hyperbolic-interior}
\]
Using \(|r+s|^2\ge \frac12|r|^2-|s|^2\) on the left collar gives
\[
\begin{aligned}
    \int_0^\varepsilon |v(x)|^2\,dx
    &\ge
    \frac{1-e^{-2k\varepsilon}}{4k}|q|^2
    -\varepsilon e^{-2k(a-\varepsilon)}|p|^2,
\end{aligned}
\]
and the analogous estimate on the right collar gives
\[
\begin{aligned}
    \int_{a-\varepsilon}^a |v(x)|^2\,dx
    &\ge
    \frac{1-e^{-2k\varepsilon}}{4k}|p|^2
    -\varepsilon e^{-2k(a-\varepsilon)}|q|^2.
\end{aligned}
\]
Hence,
\[
    \|v\|_{L^2(I_\omega)}^2
    \ge
    \left[
        \frac{1-e^{-2k\varepsilon}}{4k}
        -\varepsilon e^{-2k(a-\varepsilon)}
    \right]
    \bigl(|p|^2+|q|^2\bigr).
\]
There exist \(K_1>0\) and \(c_\varepsilon>0\), depending only on
\(a\) and \(\varepsilon\), such that for \(k\ge K_1\),
\[
    \|v\|_{L^2(I_\omega)}^2
    \ge
    \frac{c_\varepsilon}{k}
    \bigl(|p|^2+|q|^2\bigr).
\tag{A.4}\label{eq:two-collar-hyperbolic-collar}
\]
Combining \eqref{eq:two-collar-hyperbolic-interior} and
\eqref{eq:two-collar-hyperbolic-collar}, and using
\[
    \sup_{k\ge K_1} k e^{-2kd_0}<\infty,
\]
yields the desired uniform estimate for the large hyperbolic regime.
Let
\[
    K:=\max\{2/\varepsilon,K_1\}.
\]
It remains to treat the bounded transition range \(0\le |\alpha|^{1/2}\le K\),
including \(\alpha=0\). For the oscillatory family use the continuous basis
\[
    \phi_{1,k}(x)=\cos(kx),
    \qquad
    \phi_{2,k}(x)=
    \begin{cases}
        \dfrac{\sin(kx)}{k},&k>0,\\[1ex]
        x,&k=0,
    \end{cases}
\]
and for the hyperbolic family, use
\[
    \psi_{1,k}(x)=\cosh(kx),
    \qquad
    \psi_{2,k}(x)=
    \begin{cases}
        \dfrac{\sinh(kx)}{k},&k>0,\\[1ex]
        x,&k=0.
    \end{cases}
\]
For either family, let \(G_\omega(k)\) denote the \(2\times2\) Hermitian Gram
matrix of the corresponding basis in \(L^2(I_\omega)\). For each fixed
\(k\), this matrix is positive definite: otherwise a nonzero solution would
vanish on the nonempty open set \(I_\omega\), and uniqueness for the
second-order ODE would force it to vanish identically. The entries of
\(G_\omega(k)\) depend continuously on \(k\), including at \(k=0\).
Therefore, on every bounded interval \(0\le k\le K\),
\[
    \inf_{0\le k\le K}\lambda_{\min}(G_\omega(k))>0.
\]
The corresponding Gram matrices on \(I_0\) have uniformly bounded largest
eigenvalue on the same compact parameter interval. It follows that
\[
    \frac{\|v\|_{L^2(I_0)}}{\|v\|_{L^2(I_\omega)}}
\]
is uniformly bounded throughout the transition regime. Combining the three
cases proves the lemma.
\end{proof}

\begin{proposition}[Fixed-collar control in a periodic cylinder]
\label{prop:periodic-cylinder-control}
Let
\[
    C\simeq(0,a)_x\times(\mathbb R/L\mathbb Z)_y
\]
be a maximal flat cylinder, with the \(y\)-direction parallel to its periodic
trajectories. Let \(C_0\Subset C\), and suppose that \(\omega\subset C\)
contains collars
\[
    \bigl((0,\varepsilon)\cup(a-\varepsilon,a)\bigr)
    \times(\mathbb R/L\mathbb Z)
\]
for some \(0<\varepsilon<a/2\).\\
Let
\[
    E_\theta=e(hD_y),
    \qquad
    \widetilde E_\theta=\widetilde e(hD_y),
\]
where \(e,\widetilde e\in C_c^\infty(\mathbb R)\) are fixed and
\[
    \widetilde e=1
    \qquad\text{on a neighborhood of }\operatorname{supp}e.
\]
Then, there exists \(C_{C_0,\omega}>0\), independent of \(h\), such that every
exact solution
\[
    P_hu_h=0
\]
on \(C\) satisfies
\[
    \|E_\theta u_h\|_{L^2(C_0)}
    \le
    C_{C_0,\omega}
    \|\widetilde E_\theta u_h\|_{L^2(\omega)}.
\]
\end{proposition}

\begin{proof}
In the flat cylinder coordinates,
\[
    P_h=-h^2(\partial_x^2+\partial_y^2)-1.
\]
Since \(E_\theta=e(hD_y)\) is a Fourier multiplier in the periodic variable,
\[
    [P_h,E_\theta]=0.
\]
Hence, \(v_h:=E_\theta u_h\) is again an exact solution of \(P_hv_h=0\).
Choose an interval \(I_0\Subset(0,a)\) containing the transverse projection
of \(C_0\). Expanding
\[
    v_h(x,y)
    =
    \sum_{n\in\mathbb Z}
    v_{n,h}(x)e^{2\pi iny/L},
\]
each Fourier coefficient satisfies
\[
    -h^2v_{n,h}''+
    \left[
        \left(\frac{2\pi hn}{L}\right)^2-1
    \right]v_{n,h}=0.
\]
Applying \cref{lem:uniform-two-collar} to every mode, with a constant
independent of \(n\) and \(h\), and then summing by Parseval gives
\[
    \|E_\theta u_h\|_{L^2(C_0)}
    \le
    C_{C_0,\omega}
    \|E_\theta u_h\|_{L^2(\omega)}.
\]
Finally, \(\widetilde e=1\) on \(\operatorname{supp}e\), so
\[
    E_\theta=E_\theta\widetilde E_\theta.
\]
Because the collar contains complete periodic fibers, the \(L^2\)-boundedness
of the Fourier multiplier \(E_\theta\), applied fiberwise in \(y\), yields
\[
    \|E_\theta u_h\|_{L^2(\omega)}
    \le
    \|e\|_{L^\infty}
    \|\widetilde E_\theta u_h\|_{L^2(\omega)}.
\]
Combining the two estimates proves the proposition.
\end{proof}

\begin{remark}[Dependence on the collar]
The constant \(C_{C_0,\omega}\) is associated with a fixed cylinder, compact
core, collar, and directional localization. No uniformity is asserted as the
collar width tends to zero. Thus the proposition transfers regular cylinder
control to a fixed neighborhood of the saddle-connection boundary but does
not by itself control concentration on the limiting saddle/conic network. The
latter is the additional quantitative mechanism isolated in
\cref{def:qsn-exclusion}.
\end{remark}





\begin{thebibliography}{10}
\expandafter\ifx\csname url\endcsname\relax
  \def\url#1{\texttt{#1}}\fi
\expandafter\ifx\csname urlprefix\endcsname\relax\def\urlprefix{URL }\fi
\expandafter\ifx\csname href\endcsname\relax
  \def\href#1#2{#2} \def\path#1{#1}\fi

\bibitem{MarklofRudnick2012}
J.~Marklof, Z.~Rudnick, Almost all eigenfunctions of a rational polygon are
  uniformly distributed, Journal of Spectral Theory 2~(1) (2012) 107--113.
\newblock \href {https://doi.org/10.4171/JST/23} {\path{doi:10.4171/JST/23}}.

\bibitem{HippiMikkelsen2026}
K.~Hippi, S.~Mikkelsen, Quantum mixing for eigenfunctions of rational polygons
  in configuration space, arXiv:2608.11728 (2026).
\newblock \href {http://arxiv.org/abs/2608.11728} {\path{arXiv:2608.11728}}.

\bibitem{FriedlandUeberschaer}
O.~Friedland, H.~Uebersch\"ar, Scarred quasimodes on translation surfaces,
  arXiv preprint arXiv:1812.08467 (2018).
\newblock \href {http://arxiv.org/abs/1812.08467} {\path{arXiv:1812.08467}}.

\bibitem{BurqZworski2019}
N.~Burq, M.~Zworski, Rough controls for schr\"odinger operators on 2-tori,
  Annales Henri Lebesgue 2 (2019) 331--347.
\newblock \href {https://doi.org/10.5802/ahl.19} {\path{doi:10.5802/ahl.19}}.

\bibitem{BurqGermainSorellaZhu2026}
N.~Burq, P.~Germain, M.~Sorella, H.~Zhu, Trace and observability inequalities
  for laplace eigenfunctions on the torus, Forum of Mathematics, Sigma 14
  (2026) e47.
\newblock \href {https://doi.org/10.1017/fms.2026.10199}
  {\path{doi:10.1017/fms.2026.10199}}.

\bibitem{AlphonseLafontaine2025}
P.~Alphonse, D.~Lafontaine, Observability estimates for the schr{\"o}dinger
  equation on the equilateral triangle, Annales de la Facult{\'e} des Sciences
  de ToulouseTo appear; arXiv:2509.24642 (2025).
\newblock \href {http://arxiv.org/abs/2509.24642} {\path{arXiv:2509.24642}}.

\bibitem{Siudeja2016}
B.~Siudeja, On mixed dirichlet--neumann eigenvalues of triangles, Proceedings
  of the American Mathematical Society 144~(6) (2016) 2479--2493.
\newblock \href {https://doi.org/10.1090/proc/12888}
  {\path{doi:10.1090/proc/12888}}.

\bibitem{Stempak2022}
K.~Stempak, The laplacian with mixed dirichlet--neumann boundary conditions on
  weyl chambers, Journal of Differential Equations 329 (2022) 348--370.
\newblock \href {https://doi.org/10.1016/j.jde.2022.05.005}
  {\path{doi:10.1016/j.jde.2022.05.005}}.

\bibitem{HassellHillairetMarzuola2009}
A.~Hassell, L.~Hillairet, J.~Marzuola, Eigenfunction concentration for
  polygonal billiards, Communications in Partial Differential Equations 34~(5)
  (2009) 475--485.
\newblock \href {https://doi.org/10.1080/03605300902768909}
  {\path{doi:10.1080/03605300902768909}}.

\bibitem{CekicGeorgievMukherjee2020}
M.~Ceki\'c, B.~Georgiev, M.~Mukherjee, Polyhedral billiards, eigenfunction
  concentration and almost periodic control, Communications in Mathematical
  Physics 377 (2020) 2451--2487.
\newblock \href {https://doi.org/10.1007/s00220-020-03741-0}
  {\path{doi:10.1007/s00220-020-03741-0}}.

\bibitem{Yang2022ProductCones}
M.~Yang, Propagation of polyhomogeneity, diffraction, and scattering on product
  cones, Journal of Spectral Theory 12~(2) (2022) 381--415.
\newblock \href {https://doi.org/10.4171/JST/404} {\path{doi:10.4171/JST/404}}.

\bibitem{Hintz2024Cone}
P.~Hintz, Semiclassical propagation through cone points, Analysis \& PDE
  17~(10) (2024) 3477--3550.
\newblock \href {https://doi.org/10.2140/apde.2024.17.3477}
  {\path{doi:10.2140/apde.2024.17.3477}}.

\bibitem{GutkinJudge2000}
E.~Gutkin, C.~Judge, Affine mappings of translation surfaces: geometry and
  arithmetic, Duke Mathematical Journal 103~(2) (2000) 191--213.
\newblock \href {https://doi.org/10.1215/S0012-7094-00-10321-3}
  {\path{doi:10.1215/S0012-7094-00-10321-3}}.

\bibitem{Veech1989}
W.~A. Veech, Teichm{\"u}ller curves in moduli space, eisenstein series and an
  application to triangular billiards, Inventiones Mathematicae 97~(2) (1989)
  553--584.
\newblock \href {https://doi.org/10.1007/BF01388890}
  {\path{doi:10.1007/BF01388890}}.

\bibitem{SmillieWeiss2008}
J.~Smillie, B.~Weiss, Veech's dichotomy and the lattice property, Ergodic
  Theory and Dynamical Systems 28~(6) (2008) 1959--1972.

\bibitem{MelroseWunsch2004}
R.~B. Melrose, J.~Wunsch, Propagation of singularities for the wave equation on
  conic manifolds, Inventiones Mathematicae 156~(2) (2004) 235--299.
\newblock \href {https://doi.org/10.1007/s00222-003-0339-y}
  {\path{doi:10.1007/s00222-003-0339-y}}.

\bibitem{BurqZworskiPartiallyRectangular}
N.~Burq, M.~Zworski, Eigenfunctions for partially rectangular billiards,
  Preprint (2003).
\newblock \href {http://arxiv.org/abs/math/0312098}
  {\path{arXiv:math/0312098}}.
  

\bibitem{Mar_2007}
Jeremy Marzuola, Eigenfunctions for Partially Rectangular Billiards, Communications in Partial Differential Equations 31~(5) (2006) 775--790.
\newblock \href {https://doi.org/10.1080/03605300500532939}
  {\path{doi:10.1080/03605300500532939}}.


\end{thebibliography}
\end{document}